\documentclass{amsart}
\usepackage[utf8]{inputenc}
\usepackage{amsmath,amssymb,amsthm, xcolor, enumitem, comment, todonotes}
\usepackage{url}
\usepackage{hyperref}
\usepackage{lipsum}

\newcommand\blfootnote[1]{%
  \begingroup
  \renewcommand\thefootnote{}\footnote{#1}%
  \addtocounter{footnote}{-1}%
  \endgroup
}

\theoremstyle{definition}
\newtheorem{definition}{Definition}[section]

\newtheorem{example}[definition]{Example}

\newtheorem{question}[definition]{Question}

\newtheorem{notation}[definition]{Notation}
\newtheorem{claim}[definition]{Claim}

\newtheorem{remark}[definition]{Remark}

\newcommand{\mo}{\triangleleft}

\newcommand{\fr}{{}^\frown}

\newcommand{\name}{\dot}

\newcommand{\la}{\langle}

\newcommand{\ra}{\rangle}

\newcommand{\elem}{\prec}
\newcommand{\uhr}{\restriction}

\newcommand{\ol}{\ol}
\newcommand{\po}{\mathbb{P}}

\renewcommand{\ol}{\bar}

\newcommand{\MS}{\mathcal{MS}}

\newcommand{\mc}[2]{mc_{#1}(#2)}

\DeclareMathOperator{\dom}{dom}

\DeclareMathOperator{\suc}{succ}

\DeclareMathOperator{\add}{Add}

\DeclareMathOperator{\Ord}{Ord}

\DeclareMathOperator{\cp}{cp}

\DeclareMathOperator{\bd}{bd}

\newcommand{\dhr}{\downharpoonright}

\theoremstyle{plain}
\newtheorem{theorem}[definition]{Theorem}
\newtheorem{proposition}[definition]{Proposition}
\newtheorem{lemma}[definition]{Lemma}
\newtheorem{corollary}[definition]{Corollary}

\title[Approximating diamonds at an inaccessible cardinal]{Approximating diamond principles on products at an inaccessible cardinal}

\author{Omer Ben-Neria
 and Jing Zhang 
}

\date{\today}

\begin{document}

\maketitle

\blfootnote{
The first author was partially supported by the Israel Science Foundation (Grant 1832/19). \\
The second author was supported by the European Research Council (grant agreement ERC-2018-StG 802756).\\
The authors would like to thank Spencer Unger and Miha Habic for valuable discussions and comments on this work. 
\par\nopagebreak
\textit{E-mails}: \texttt{omer.bn@mail.huji.ac.il}, \texttt{jingzhan@alumni.cmu.edu}
}

\begin{abstract}
We isolate \emph{the approximating diamond principles}, which are consequences of the diamond principle at an inaccessible cardinal.
We use these principles to find new methods for negating the diamond principle at large cardinals. Most notably, we demonstrate, using Gitik's overlapping extenders forcing, a new method to get the consistency of the failure of the diamond principle at a large cardinal $\theta$ without changing cofinalities or adding fast clubs to $\theta$. In addition, we show that the approximating diamond principles necessarily hold at a weakly compact cardinal. This result, combined with the fact that in all known models where the diamond principle fails the approximating diamond principles also fail at an inaccessible cardinal, exhibits essential combinatorial obstacles to make the diamond principle fail at a weakly compact cardinal. 
\end{abstract}

\section{Introduction}

Fix a regular uncountable cardinal $\theta$ and a stationary set $S\subset \theta$. A sequence $\langle x_\alpha\subset \alpha: \alpha\in S\rangle$ is a \emph{$\diamondsuit(S)$-sequence} if for any $X\subset \theta$, the set $\{\alpha\in S: X\cap \alpha = x_\alpha\}$ is stationary. This principle is due to Jensen \cite{MR309729} and was discovered during the course of his fine structural analysis of the constructible universe. Note that $\diamondsuit(\theta)$ implies $\theta^{<\theta}=\theta$.

The history of the relation between the diamond and compactness principles, goes back to the work of Kunen and Jensen \cite{JensenKunen}, who showed that  $\diamondsuit(\theta)$ must hold at every subtle cardinal. 
In fact, they prove that the stronger property $\diamondsuit(Reg^\theta)$ holds at such cardinals, where $Reg^\theta$ is the stationary collection of regular cardinals below $\theta$.

The consistency of $\neg\diamondsuit(Reg^\theta)$ along with $\theta$ being weakly compact, $\Pi^{m}_n$-indescribable, strongly unfolddable
was established by Woodin, Hauser \cite{MR1164732} and D\v{z}amonja-Hamkins \cite{MR2279655} respectively. Each of these consistency results is established from its minimal corresponding  large cardinal assumption.

In contrast, $\neg\diamondsuit(\theta)$ at Mahlo $\theta$, is known to have a significantly stronger consistency strength. 
 Jensen \cite{Jensen69} has shown that $\neg\diamondsuit(\theta)$ 
at a Mahlo cardinal $\theta$ implies the existence of $0^\#$.  Zeman \cite{Zeman00} improved
the lower bound to the existence of an inner model with a cardinal $\theta$, in which for every $\gamma < \theta$, the set
$\{ \alpha < \theta \mid o(\alpha) \geq \gamma\}$ is stationary in $\theta$.

There are also tight connections between the diamond principles and cardinal arithmetic. Shelah \cite{MR2596054}, building on the previous work by Gregory \cite{MR485361} and after a series of partial results, showed that for any uncountable cardinal $\lambda$, $2^\lambda=\lambda^+$ iff $\diamondsuit(\lambda^+)$. On the other hand, it was previously known from work of Jensen \cite{MR0384542} that CH does not imply $\diamondsuit(\omega_1)$. See \cite{MR2777747} for a comprehensive survey regarding the diamond principle at successor cardinals. The situation at inaccessible cardinals is different. In \cite{MR1998104}, Shelah showed that it is consistent with GCH that $\theta$ is inaccessible and $\neg \diamondsuit(S)$ for some stationary $S$ such that $\theta-S$ is fat, from the minimal assumption. 
Works of Gitik (see for example \cite{MR882254}) demonstrate the consistency of GCH along with the saturation of $\mathrm{NS}_{\theta}\restriction S$, hence $\neg\diamondsuit(S)$, where $S$ is also co-fat. However, in all of these scenarios, the diamond principle only fails partially. In other words, $\diamondsuit(\theta)$ holds in all of these models mentioned above.

Two of the main problems in this area are: 
\begin{enumerate}
\item Is it consistent that $\diamondsuit(\theta)$ fails at a weakly compact $\theta$?
\item Is it consistent with GCH that $\diamondsuit(\theta)$ fails?
\end{enumerate}

The only previously known method to violate the diamond principle fully at an inaccessible cardinal is due to Woodin \cite{CummingsWoodin}, based on the analysis of Radin forcing extensions. This approach has been further developed in \cite{MR3960897} and \cite{JingOmerRadin} to obtain the failure of the full diamond at a stationary reflecting cardinal and even at a cardinal carrying no amenable C-sequences.  

Nevertheless, this method is not applicable for solving the two problems above, due to the following. In the following, when we say ``Radin forcing at $\theta$'', we mean ``Radin forcing defined from a measure sequence on $\theta$''.

\begin{enumerate}
\item It was shown in \cite{JingOmerRadin} that in any Radin forcing extension at $\theta$, if $\theta$ is weakly compact, then $\diamondsuit(\theta)$ holds.
\item It was shown in \cite{CummingsMagidor} that in the Radin forcing extension at $\theta$ over a model of GCH, if $\theta$ is regular, then $\diamondsuit(\theta)$ holds.
\end{enumerate}

These results motivate the need for new methods to make the diamond principle fail at an inaccessible cardinal.

In order to state the main results of this paper, we need some definitions. 

\begin{definition}\label{definition:SimplifiedDomination}
Let $\delta$ be a limit ordinal and $f,g: \delta\to \Ord$ be functions. We write $f<^* g$ if $\{\alpha<\delta: f(\alpha)<g(\alpha)\}$ contains a closed unbounded set in $\delta$ if $cf(\delta)>\omega$ and contains co-bounded set if $cf(\delta)=\omega$.
\end{definition}

Let $\theta$ be an inaccessible cardinal, $\vec{\lambda}=\langle \lambda_i<\theta: i<\theta\rangle$ be an increasing sequence of cardinals cofinal in $\theta$ and $S\subset \theta$ be a stationary set. We define what it means to be an \emph{approximating diamond sequence on $\Pi_{i<\theta}\lambda_i$} with support $S$.

\begin{definition}\label{definition: SimplifiedDiamond_lambda}
We say $\langle g_\gamma\in \Pi_{i<\gamma}\lambda_i: \gamma\in S\rangle$ is a $\diamondsuit_{\vec{\lambda}}(S)$-sequence if for any $f\in \Pi_{i<\theta} \lambda_i$, the collection $\{\gamma\in S: g_\gamma \not <^* f\restriction \gamma\}$ is stationary.
\end{definition}

 The principle can be considered as the inaccessible analogue of the \emph{parametrized diamond principle} corresponding to the bounding number $\mathfrak{b}$ studied by 
Moore, Hru\v{s}\'{a}k and D\v{z}amonja \cite{MR2048518} at the level of $\omega_1$.\\

Easy observation reveals that for every sequence $\vec{\lambda} = \la \lambda_i : i < \theta\ra$, $\diamondsuit_{\vec{\lambda}}(S)$ follows from $\diamondsuit(S)$ (see Lemma \ref{lemma: diamondimplies}), suggesting to force  $\neg\diamondsuit_{\vec{\lambda}}(\theta)$ as a new route for obtaining $\neg\diamondsuit(\theta)$ at a large cardinal $\theta$. 
Indeed, we show in Theorem \ref{Thm:ApproxDiaomondFailsInRadin} that $\diamondsuit_{\vec{\beth}}(\theta)$ fails in Woodin's model \cite{CummingsWoodin}, where $\vec{\beth}$ denote the sequence $\langle 2^\alpha: \alpha<\theta\rangle$.
A similar argument shows that $ \diamondsuit_{\vec{\beth}}(\theta)$ fails in various Extender-based Radin forcing of Merimovich \cite{MerimovichEBF}.
More surprising to the authors was that one can force the failure of $\diamondsuit_{\vec{\beth}}(\theta)$ without changing cofinalities or adding fast clubs. This is fundamentally different from Woodin's method.
To obtain this result we  
utilize a version of the recent forcing discovered by Gitik \cite{GITIK:OverlappingEBF} using overlapping extenders.

\begin{theorem}\label{theorem: mainconsistency}
Suppose $\theta$ is an inaccessible limit of $\theta^{++}$-strong cardinals, then there exists a cofinality and cardinal preserving forcing extension $V[G]$ of $V$ in which: 
\begin{enumerate}
\item $\diamondsuit_{\vec{\beth}}(\theta)$ fails,
\item $V,V[G]$ agree on cardinals and cofinalities,  and
\item any closed unbounded subset of $\theta$ in $V[G]$ contains a closed unbounded subset of $\theta$ from $V$.
\end{enumerate}
\end{theorem} 

The third condition 
in Theorem \ref{theorem: mainconsistency} holds
suggests that $V[G]$ is a relatively tame forcing extension of $V$ in some sense and will be useful in proving certain compactness properties of $\theta$ in $V$ remain valid in $V[G]$.
This reasoning is shown to be true to an extent in Section \ref{Section: Compactness}.
However, this approach of forcing $\neg\diamondsuit(\theta)$ must not succeed at the level of weak compactness, as evidenced by the following theorem. 
\begin{theorem}\label{theorem: mainWC}
If $\theta$ is weakly compact, then $\diamondsuit_{\vec{\lambda}}(Reg^\theta)$ holds for any increasing sequence $
\vec{\lambda}$ in $\theta$. Moreover, $\Diamond_{\vec{\lambda}}(S)$ holds for every $S$ in the weak compact filter on $\theta$. 
\end{theorem}

The last theorem prompts the investigation of a possible implication from $\diamondsuit_{\vec{\lambda}}(S)$ to $\diamondsuit(\theta)$, for some sequence $\vec{\lambda}$ and stationary set $S \subset \theta$.  
We do not know of such implication in general but are able to prove an implication under GCH, in the case $S \subset Sing^\theta_{>
\omega}$, where $Sing^\theta_{>
\omega}$ is the set of singular ordinals of uncountable cofinalities. 
\begin{theorem}\label{theorem: mainGCH}
 $\diamondsuit_{\vec{\beth}}(Sing^\theta_{>
\omega})$ and $\diamondsuit(Sing^\theta_{>
\omega})$ are equivalent under GCH. \\
\end{theorem}

The organization of the paper is: 
\begin{enumerate}
\item Section \ref{Section:ZFC} contains the proofs of Theorem \ref{theorem: mainWC} and  \ref{theorem: mainGCH} as well as some generalizations.
\item Section \ref{Section: Radin} outlines a few results in the Radin forcing extension. In particular, the consistency of $\neg\diamondsuit_{\vec{\lambda}}(\theta)$ and the consistency of $\diamondsuit_{\vec{\beth}}(Sing^\theta_{>\omega}) + \neg \diamondsuit(\theta)$ are demonstrated.
\item Section \ref{Section: forcing} gives a short description of Gitik's forcing with overlapping extenders of inaccessible length 
and review some of its key properties.
\item Section \ref{Section: failure} focuses on the proof of Theorem \ref{theorem: mainconsistency}. It also includes a proof showing that the assumption of the extenders being strong past $\theta$ is needed. 

\item Section \ref{Section: Compactness} illustrates the possibilities of preserving certain compactness properties of $\theta$ from the ground model, extending some results from \cite{JingOmerRadin}.

\item Section \ref{Section: Questions} concludes with open questions.
\end{enumerate}





    


\section{ZFC Results}\label{Section:ZFC}
Let $\theta$ be a strongly inaccessible cardinal.

\begin{definition}
Let $J$ be an ideal on $\gamma$ and $f,g\in {}^\gamma \Ord$, then $f <_J g$ abbreviates $\{\alpha\in \gamma: f(\alpha)\geq g(\alpha)\}\in J$. Similary we can define $\leq_J, =_J$.
\end{definition}
Unless otherwise noted, all the ideals we consider are proper.

\begin{notation}
Given an ideal $J$, let 

\begin{itemize}
\item $J^+$ denote the collection of positive sets, namely $J^+=\{A: A\not\in J\}$ and
\item $J^*$ denote its dual filter, namely $J^*=\{A: A^c\in J\}$.
\end{itemize} 
\end{notation}

\begin{definition}
Suppose we are given 
\begin{enumerate}
    \item $\vec{J}=\langle J_\gamma\subset P(\gamma): \gamma \leq \theta\rangle$ where each $J_\gamma$ is an ideal on $\gamma$,
    \item an increasing sequence $\vec{\lambda}=\langle \lambda_i: i<\theta\rangle$ cofinal in $\theta$, and
    \item $S\in J_\theta^+$.
\end{enumerate} 
Then a \emph{$\diamondsuit_{\vec{J},\vec{\lambda}}(S)$-sequence} is $\langle g_\gamma\in \Pi_{i<\gamma}\lambda_i: \gamma\in S\rangle$ satisfying: for any $f\in \Pi_{i<\theta}\lambda_i$, $\{\gamma\in S: g_\gamma\not <_{J_{\gamma}} f\restriction \gamma\}\in J_\theta^+$.
\end{definition}
We say $\diamondsuit_{\vec{J},\vec{\lambda}}(S)$ holds if there exists a $\diamondsuit_{\vec{J},\vec{\lambda}}(S)$-sequence. In this paper, unless noted otherwise, we will assume the index set $S$ is a subset of the club set $\{\gamma<\theta: \forall i<\gamma, \lambda_i<\gamma\}$. 

\begin{lemma}\label{lemma: diamondimplies}
$\diamondsuit(S)$ implies $\diamondsuit_{\vec{J}, \vec{\lambda}}(S)$ for any $\vec{J}, \lambda, S\in \mathrm{NS}_\theta^+$ satisfying that $J_\theta \subseteq \mathrm{NS}_\theta$.
\end{lemma}
\begin{proof}
Let $\langle g'_\gamma\in {}^\gamma \gamma: \gamma\in S\rangle$ be a given $\diamondsuit(S)$-sequence. Define a $\diamondsuit_{\vec{J}, \vec{\lambda}}(S)$ by letting $g_\gamma=g'_\gamma$ if $g'_\gamma\in \Pi_{i<\gamma}\lambda_i$, and $g_\gamma\equiv 0$ otherwise. Given $f\in \Pi_{i<\theta} \lambda_i$, there exists a stationary set $S'\subset S$ such that for all $\gamma\in S'$, $f\restriction \gamma = g'_\gamma = g_\gamma$, where the last equality follows from the definition of $g_\gamma$. As $J_\theta \subseteq \mathrm{NS}_\theta$, $S'\in J_\theta^+$.
\end{proof}

\begin{definition}
A transitive model $M\models \mathrm{ZFC}^-$ ($\mathrm{ZFC}-$ Power Set Axiom) is a \emph{$\theta$-model} if it contains $\theta$ and ${}^{<\theta} M\subset M$. 
\end{definition}

A typical $\theta$-model will be the transitive collapse of some $X\prec H(\chi)$ where $\chi$ is a large enough regular cardinal, $\theta+1\subset X$ and ${}^{<\theta} X\subset X$.

\begin{definition}
A sequence of ideals $\vec{J}=\langle J_\gamma: \gamma\leq \theta\rangle$ is 
\begin{enumerate}
    \item \emph{uniformly definable} if there exist a formula $\varphi(x,y)$ and a parameter $a\subset \theta$, such that for any large enough $\gamma$, $A\in J_\gamma$ iff $H(\gamma^+)\models \varphi(A, a\cap \gamma)$.
    \item \emph{$\theta$-upward uniformly definable} if it is uniformly definable as witnessed by $\varphi(x, y), a$ and for any $\theta$-model $N$ containing $a$, $(J_\theta)^N\subset J_\theta$. More precisely, the latter asserts for any $A\in N\cap P(\theta)$, if $N\models`` H(\theta^+)\models \varphi(A, a) "$, then $H(\theta^+)\models \varphi(A, a)$, namely, $A\in J_\theta$.
\end{enumerate}
\end{definition}

\begin{remark}
Usually it suffices to work with $\langle J_\gamma: \gamma\in A\cup \{\theta\}\rangle$ where $A\in J_\theta^*$. Hence, for the rest of this article, whenever we define a sequence of ideals on a large subset of $\theta$, we implicitly mean that any ideal can be filled in for missing coordinates.
\end{remark}

\begin{example}
We give two concrete examples of $\theta$-upward uniformly definable sequence of ideals: 
\begin{enumerate}
\item $\vec{J}_{bdd}=_{\mathrm{def}}\langle P_{bdd}(\gamma): \gamma\in \lim \theta\rangle$, where $P_{bdd}(\gamma)$ is the collection of all bounded subsets of $\gamma$. Let $\varphi(x)$ be ``there exists $y$ which is the largest cardinal and there exists $\eta<y$ such that for all $\gamma'>\gamma, \gamma'\not\in x$''. Then we see that $A\in J_\gamma$ iff $H(\gamma^+)\models \varphi(A)$. It is easy to check that this definition is upward absolute for $\theta$-models.

\item $\vec{\mathrm{NS}}=\langle \mathrm{NS}_{\gamma}: \gamma\in \theta\cap \mathrm{cof}(>\omega)\rangle$, where $\mathrm{NS}_\gamma$ is the nonstationary ideal on $\gamma$. Let $\varphi(x)$ be ``there exists $y$ which is the largest cardinal and there exists a closed unbounded $c\subset y$ such that $c\cap x=\emptyset$''. Then we see that $A\in J_\gamma$ iff $H(\gamma^+)\models \varphi(A)$. The key point to see that this definition is upward absolute for $\theta$-models is that for any $\theta$-model $N$, if $C\in N$ such that $N\models C$ is club in $\theta$, then $H(\theta^+)\models C$ is club in $\theta$.
\end{enumerate}
\end{example}

Note that $\diamondsuit_{\vec{\lambda}}(S)$ is exactly $\diamondsuit_{\vec{J}^*, \vec{\lambda}}(S)$ where  $J^*_\gamma = NS_\gamma$ when $cf(\gamma) > \omega$ and $J^*_\gamma = P_{bdd}(\gamma)$ otherwise.
Using $\vec{J}^*$, the following Theorem implies Theorem \ref{theorem: mainWC}.  
\begin{theorem}\label{theorem:weakcompactdiamond}
If $\theta$ is weakly compact, then $\diamondsuit_{\vec{J}, \vec{\lambda}}(S)$ holds when
\begin{enumerate}
    \item $\vec{\lambda}=\la \lambda_i: i<\theta\ra$ is increasing and cofinal in $\theta$,
    \item $\vec{J}$ is $\theta$-upward uniformly definable,
    \item $J_\theta$ is $\sigma$-complete,
    \item $J_\theta \subset \mathrm{WC}_\theta$, where $\mathrm{WC}_\theta$ is the weakly compact ideal on $\theta$,
    \item $S\in \mathrm{WC}_\theta^+$.
\end{enumerate}
\end{theorem}
\begin{proof}
Suppose for the sake of contradiction that the conclusion fails. More explicitly, this means for any sequence $\langle g_\gamma\in \Pi_{i<\gamma}\lambda_i: \gamma\in S\rangle$, there exist $f\in \Pi_{i<\theta}\lambda_i$ and $A\in J^*_\theta$ such that for all $\gamma\in A\cap S$,  $g_\gamma<_{J_\gamma} f\restriction \gamma$. Note that by the assumption that $J_\theta \subset \mathrm{WC}_\theta$, $A\in \mathrm{WC}_\theta^*$. Let $\varphi(x,y)$ and $a\subset \theta$ witness the $\theta$-upward uniform definability of $\vec{J}$.

Recursively define $\langle g_\gamma\in \Pi_{i<\gamma}\lambda_i: \gamma\in \lim (\theta+1)\rangle$ such that for each $\gamma\in \lim \theta+1$, if there exists $g\in \Pi_{i<\gamma} \lambda_i$ such that:
\begin{enumerate}
\item for $J_\gamma^*$-many $\eta\in S\cap \gamma$, $g\restriction \eta>_{J_\eta} g_{\eta}$ and 
\item there does not exist $f<_{J_\gamma} g$ satisfying the property above,
\end{enumerate} 
then $g_\gamma$ is any such $g$. Otherwise, $g_\gamma\equiv 0$ and we say $g_\gamma$ is \emph{defined vacuously}.

\begin{claim}\label{claim:uncountable}
If $J_\gamma$ is $\sigma$-complete and there exists $g$ satisfying (1), then there exists $g$ satisfying (2).
\end{claim}
\begin{proof}[Proof of the Claim]
Suppose for the sake of contradiction, no $g$ satisfies (2). Define $g_0>_{J_\gamma} g_1>_{J_\gamma}g_2 >_{J_\gamma}\cdots g_n>_{J_\gamma}\cdots $ with each $g_n$ satisfying (1). Since $J_\gamma$ is $\sigma$-complete, there exists $\gamma'<\gamma$ such that for any $m<n$, $g_m(\gamma') > g_n(\gamma')$. Contradiction.
\end{proof}
Note that by Claim \ref{claim:uncountable} and the initial assumption, $g_\theta$ is not defined vacuously. Therefore, there exists $A\in J_\theta^*$, such that for any $\tau\in A\cap S$, $g_\tau<_{J_\tau} g_{\theta}\restriction \tau$.

As $A\cap S\in \mathrm{WC}^+_\theta$, we can find  
\begin{itemize}
\item a transitive $\theta$-model $M$ containing $\vec{g}=\langle g_\gamma: \gamma \in S\cup \{\theta\}\rangle$, $\vec{J}\restriction \theta$ and $a$,
\item a transitive $\theta$-model $N$ and an elementary embedding $j: M\to N$ with critical point $\theta$,
\item $\theta\in j(A\cap S)$.
\end{itemize}
Let $f_\theta = j(\vec{g})(\theta)$. Observe that $f_\theta$ is not defined vacuously since $j(g_\theta)\restriction \theta =g_\theta\in N$ is a witness to property (1) and Claim \ref{claim:uncountable} implies there is a function in $N$ satisfying both (1) and (2). The key point is: due to the fact that $(J_\theta)^N\subset J_\theta$, if we can find an infinite $<_{(J_\theta)^N}$-decreasing sequence, then it is an infinite $<_{J_\theta}$-decreasing sequence, contradicting with Claim \ref{claim:uncountable}. By the elementarity of $j$ and the fact that $\theta\in j(S\cap A)$, we have that $N\models f_\theta<_{J_\theta} j(g_\theta)\restriction \theta=g_\theta$. Since $\vec{J}$ is $\theta$-upward uniformly definable, we have $(J_\theta)^N \subset J_\theta$. As a result, in $V$, we have that $f_\theta<_{J_\theta}g_\theta$. This contradicts with the minimality of $g_\theta$.
\end{proof}

\begin{remark}
Under the same hypothesis as in Theorem \ref{theorem: mainWC}, the previous proof gives the following strengthening of $\diamondsuit_{\vec{J},\vec{\lambda}}(S)$, which can be seen as a symmetric version of our approximating property: there exists a sequence $\langle g_\gamma\in \Pi_{i<\gamma} \lambda_i: \gamma\in S\rangle$ such that for any $f\in \Pi_{i<\theta}\lambda_i$, the set $\{\gamma\in S: f\restriction \gamma \not <_{J_\gamma} g_\gamma \text{ and }g_\gamma\not <_{J_\gamma} f\restriction \gamma\}\in J^+_\theta$. 
\end{remark}



\begin{remark}
Similar proof shows that: if $\theta$ is a $\Pi^{1}_{n+1}$-indescribable cardinal, then for any increasing $\vec{\lambda}$ cofinal in $\theta$, there exists $\langle g_\gamma\in \Pi_{i<\gamma} \lambda_i: \gamma<\theta\rangle$ such that for any $f\in \Pi_{i<\theta} \lambda_i$, the collection $\{\gamma: g_\gamma \not <_{\mathbf{\Pi^1_n}(\gamma)} f\restriction \gamma\}$ is positive with respect to the ideal $\mathbf{\Pi^1_n}(\theta)$. Here $\mathbf{\Pi^1_n}(\gamma)$ denotes the $\Pi^1_n$-indescribable ideal on subsets of $\gamma$.
\end{remark}

\begin{remark}
Recall a theorem of Woodin (see \cite{MR1164732}) that it is consistent that $\theta$ is weakly compact and $\diamondsuit(\mathrm{Reg}^\theta)$ fails. As a result, $\diamondsuit(\mathrm{Reg}^\theta)$ is not implied by $\diamondsuit_{\vec{\mathrm{NS}}, \vec{\lambda}}(\mathrm{Reg}^\theta)$, in contrast with the following proposition.
\end{remark}

\begin{proposition}[GCH]\label{proposition:gch}
Let $S$ be a stationary subset of $S^\theta_{>\omega} = \theta\cap \{\gamma: \omega<cf(\gamma)<\gamma\}$. Then $\diamondsuit_{\vec{\mathrm{NS}},\vec{\beth}}(S)$ implies $\diamondsuit(S)$.
\end{proposition}

\begin{proof}
We may assume any $\gamma\in S$ is a singular cardinal.
Fix a $\diamondsuit_{\vec{J}, \vec{\lambda}}(S)$-sequence $\vec{g}=\langle g_\gamma: \gamma\in S\rangle$. For each $\gamma\in S$, let $c_\gamma\subset \gamma$ be a club of order type $cf(\gamma)$. For each $\gamma\in S$, let $F_{bdd}(\gamma)=\{h\in {}^{c_\gamma}\gamma: \mathrm{ran}(h)\subset \gamma'<\gamma\}$. By the cardinal arithmetic assumption, $|F_{bdd}(\gamma)|=\gamma$ for each strong limit singular cardinal $\gamma$. For each ordinal $\beta$, let $f_\beta: |\beta|\to \beta$ be some bijection and let $\langle x^\beta_{i}: i<\beta^+\rangle$ be an enumeration of $P(\beta)$.

Define $\mathcal{C}_\gamma=\{\bigcup_{\beta\in x} x^\beta_{f_{g_\gamma(\beta)+1}(h(\beta))}: x\subset c_\gamma, h\in F_{bdd}(\gamma), \dom(h)=x\}$. Note that $|\mathcal{C}_\gamma|=\gamma$. We check that $\langle \mathcal{C}_\gamma: \gamma\in S\rangle$ is a $\diamondsuit^-(S)$-sequence. Let $X\subset \theta$ be given. Let $l_X\in \Pi_{i<\theta} i^+$ be such that for each $i$, $l_X(i)=j<i^+$ where $x^i_j=X\cap i$. By the fact that $\vec{g}$ is a $\diamondsuit_{\vec{J},\vec{\lambda}}(S)$-sequence, $S'=\{\gamma\in S: \exists A \in \mathrm{NS}^+_\gamma, \forall \beta\in A, l_X(\beta) \leq g_\gamma(\beta)\}\in \mathrm{NS}_\theta^+$. It suffices to show that for each $\gamma\in S'$, $X\cap \gamma\in \mathcal{C}_\gamma$. Fix $\gamma\in S'$ as witnessed by $A\in \mathrm{NS}^+_\gamma$. Consider the following regressive function: $\beta\in A\cap c_\gamma \mapsto i_\beta<\beta$ such that $f_{g_\gamma(\beta)+1}(i_\beta)=l_X(\beta)$. By the Pressing Down lemma, there exists a stationary $x\subset c_\gamma\cap A$ and some $\gamma'<\gamma$ such that for all $\beta\in x$, $i_\beta<\gamma'$. Consider $h\in F_{bdd}(\gamma)$ defined on $x$ mapping $\beta\in x$ to $i_\beta$. Then $\bigcup_{\beta\in x} x^{\beta}_{f_{g_\gamma(\beta)+1}(h(\beta))}=\bigcup_{\beta\in x} x^{\beta}_{f_{g_\gamma(\beta)+1}(i_\beta)}=\bigcup_{\beta\in x} x^{\beta}_{l_X(\beta)}=X\cap \gamma \in \mathcal{C}_\gamma$. \end{proof}

Theorem \ref{theorem: mainGCH} is an immediate consequence of the above.
\begin{proof}(Theorem \ref{theorem: mainGCH})
$\diamondsuit(Sing^\theta_{>\omega})$ implies $\diamondsuit_{\vec{\beth}}(Sing^\theta_{>\omega})$ by Lemma \ref{lemma: diamondimplies}. Assuming GCH, 
$\diamondsuit_{\vec{\beth}}(Sing^\theta_{>\omega})$ implies
$\diamondsuit(Sing^\theta_{>\omega})$ by Proposition
\ref{proposition:gch}.
\end{proof}

\begin{remark}
Theorem \ref{theorem: SeparatingDiamonds} in the next section will show that the GCH assumption in Proposition \ref{proposition:gch} cannot be completely waived. However, the cardinal arithmetic in the model of Theorem \ref{theorem: SeparatingDiamonds} is quite specific, and it is possible that the result of Proposition \ref{proposition:gch} can be obtained from mild cardinal arithmetic assumptions such as $2^\alpha$ is regular for most $\alpha < \theta$. See problems \ref{question: approxdiamondTOdiamondWithCardArithmetic} and
\ref{question: approxdiamondEverewhereToDiamond} in the open problems section.
\end{remark}

In the rest of the section we discuss circumstances where weak compactness is strong enough to give rise to a diamond sequence. These results will be used in the next section to give a new and simple proof that diamond must hold in Radin generic extensions in which the critical point is weakly compact, which is the main result from \cite{JingOmerRadin}.

\begin{definition}
We say $\mathcal{F}\subset P(\theta)$ is 
\begin{enumerate}
    \item \emph{closed under restrictions} if for any $X\in \mathcal{F}$ and any $\alpha<\theta$, $X\cap \alpha\in \mathcal{F}$;
    \item \emph{closed under branching} if for any $X\subset \theta$ such that for each $\alpha<\theta$, $X\cap \alpha\in \mathcal{F}$, then $X\in \mathcal{F}$.
\end{enumerate}

\end{definition}

\begin{definition}
Let $\mathcal{F}\subset P(\theta)$. We say $\mathcal{F}$ admits a \emph{$\theta$-upward uniformly definable well-order} $<_\theta$ if 
\begin{enumerate}
    \item there exist a formula $\varphi(x_0, x_1, y)$ and a parameter $a\subset \theta$ such that for $A,B\in \mathcal{F}$, $A<_\theta B$ iff $H(\theta^+)\models \varphi(A,B, a)$;
    \item $<_\theta$ is a well-order;
    \item for any $\theta$-model $N$ containing $a$, $(<_\theta)^N \subset <_\theta$. More precisely, for any $A,B\in N\cap \mathcal{F}$, if $N\models ``H(\theta^+)\models \varphi(A,B, a)"$, then $H(\theta^+)\models \varphi(A,B, a)"$. 
\end{enumerate}
\end{definition}

\begin{theorem}\label{theorem:WCFullDiamond}
Suppose $\theta$ is weakly compact. Let $\mathcal{F}\subset P(\theta)$ be a collection closed under restrictions and branching admitting a $\theta$-upward uniformly definable well-order $<_\theta$. Then $\diamondsuit(\theta)$ holds for $\mathcal{F}$. Namely, there exists $\langle S_\gamma: \gamma<\theta\rangle$ such that for any $A\in \mathcal{F}$, there exists stationarily many $\gamma<\theta$ such that $A\cap \gamma=S_\gamma$.
\end{theorem}

\begin{proof}
Let $\varphi(x_0,x_1, a)$ be the $\theta$-upward uniform definition for $<_\theta$. Suppose for the sake of contradiction that there does not exist a $\diamondsuit(\theta)$-sequence for sets in $\mathcal{F}$. Define a sequence $\langle S_\gamma\subset \gamma: \gamma<\theta\rangle$ recursively as follows: at stage $\alpha$, if there exists $X\in \mathcal{F}\cap P(\alpha)$ such that 
\begin{enumerate}
\item there exists a club $c\subset \alpha$ such that for all $i\in c$, $X\cap i\neq S_i$ and 
\item there is no $Y\in \mathcal{F}\cap P(\alpha)$ with $H(\alpha^+)\models \varphi(Y,X, a\cap \alpha)$ satisfying the property above,
\end{enumerate}
then define $S_\alpha$ as one such $X$. Otherwise, just let $S_\alpha=\emptyset$. By the assumption, $\langle S_\gamma: \gamma<\theta\rangle$ as defined is not a $\diamondsuit(\theta)$-sequence for $\mathcal{F}$, as a result, there exists $X\in \mathcal{F}$ and a club $C\subset \theta$ such that for all $\gamma\in C$, $X\cap \gamma\neq S_\gamma$. Since $<_\theta$ is a well order, we may pick such $X$ such that it is $<_\theta$-least.

By the fact that $\theta$ is weakly compact, we may find $\theta$-models $M, N$ with $M$ containing $X, \langle S_\gamma: \gamma<\theta\rangle, C, \mathcal{F}'=_{def} \mathcal{F}\cap V_\theta$ and an elementary embedding $j: M\to N$ with critical point $\theta$. Let $S_\theta=j(\vec{S})(\theta)$. Since $\theta\in j(C)$, we have that $j(X)\cap \theta =X \neq S_\theta$. Note that $X\in j(\mathcal{F}')$ by the fact that every proper initial segment of $X$ is in $\mathcal{F}'$ and the elementarity of $j$, we observe that $X$ satisfies (1) with respect to $\langle S_\gamma: \gamma<\theta\rangle$ as witnessed by the club $C\subset \theta$. By the choice of $S_\theta$ and the fact that $S_\theta\neq X$, we know that $N\models`` H(\theta^+)\models \varphi(S_\theta, X, a)$''. Since $(<_\theta)^N\subset <_\theta$ by the $\theta$-upward absoluteness assumption and that $S_\theta\in \mathcal{F}$ since $\mathcal{F}$ is closed under branching, we have that $S_\theta<_\theta X$ in $V$, contradicting with the $<_\theta$-minimality of $X$.
\end{proof}

\section{Combinatorics in the Radin forcing extensions}\label{Section: Radin}
We follow \cite[Section 2]{JingOmerRadin} for the definitions and notations regarding Radin forcing. Let $\bar{U}$ be a measure sequence on $\theta$ derived from an embedding $j: V\to M$. To be consistent with notations of this paper, we note two notation differences from \cite{JingOmerRadin}.
\begin{enumerate}
\item  We use $\theta$ to denote the top cardinal that measure sequences are defined on. Namely, the measure sequences are derived from the elementary embedding $j: V\to M$ with critical point $\theta$. For a measure sequence $\bar{w}$, we then use $\theta(\bar{w})$ to denote the critical point associated with $\bar{w}$. In \cite{JingOmerRadin}, we used $\kappa$.
\item $p\geq q$ means $p$ is a stronger condition. In \cite{JingOmerRadin}, $p\leq q$ means $p$ is stronger.
\end{enumerate}

\subsection{The failure of $\diamondsuit_{\vec{J}, \vec{\lambda}}$} In this subsection, we provide an extension of Woodin's theorem that in suitable Radin extensions, $\diamondsuit(\theta)$ fails at an inaccessible $\theta$ and show that in these models, actually the weaker $\diamondsuit_{\vec{J},\vec{\lambda}}(\theta)$ fails for certain $\vec{J}$ and $\vec{\lambda}$. Then we demonstrate that in any Radin extension, the hypotheses of Theorem \ref{theorem:WCFullDiamond} holds.

For each $\alpha<\theta$, fix an enumeration $\vec{x}^\alpha=\langle x^\alpha_j: j<2^\alpha\rangle$. Let $\vec{x}^\theta = j(\langle \vec{x}^\alpha: \alpha<\theta \rangle)(\theta)$. For any $X\subset \theta$ and $\alpha\leq \theta$, let $index_{\alpha}(X)$ be the unique $j$ such that $X\cap \alpha=x^\alpha_j$. For each $X\subset \theta$ and $X\in V$, let $f_X: \theta\to \theta$ mapping $\alpha$ to $index_\alpha(X)$.

\begin{theorem}\label{Thm:ApproxDiaomondFailsInRadin}
Suppose $lh(\bar{U})<2^\theta$ and $lh(\bar{U})\in \mathrm{cof}(>\theta)$, then in $V^{R_{\bar{U}}}$, $\diamondsuit_{\vec{\mathrm{NS}}, \vec{\lambda}}(\theta)$ fails for some $\vec{\lambda}$. 
\end{theorem}

\begin{proof}
Let $G\subset R_{\bar{U}}$ be generic over $V$. Let $C_G=\{C_G(i): i<\theta\}$ is an increasing enumeration of the generic club.
Let $g: \theta \to \Ord$ be such that $j(g)(\theta)=lh(\bar{U})^+\leq 2^\theta$. We may assume that $g(\alpha)\leq 2^\alpha$ for all $\alpha<\theta$. Then $\vec{\lambda}$ will be chosen such that $\lambda_i=g(C_G(i))$.

\begin{claim}\label{claim: increasing}
For $\eta\leq lh(\bar{U})$, in $V^{R_{\bar{U}\restriction \eta}}$, for any $\alpha_0<\alpha_1<lh(\bar{U})^{+}$, $f_{x^\theta_{\alpha_0}}\restriction C_{G_\eta} <^* f_{x^\theta_{\alpha_1}}\restriction C_{G_\eta}$.
\end{claim}
\begin{proof}
Note that $\{\bar{w}: index_{\theta(\bar{w})}(x^\theta_{\alpha_0})<index_{\theta(\bar{w})}(x^\theta_{\alpha_1})\}\in \bigcap \bar{U}$, since $index_\theta(j(x^\theta_{\alpha_k}))=\alpha_{k}$ for $k<2$. As a result, the conclusion follows by genericity.
\end{proof}

\begin{claim}\label{claim: cofinal}
For $\eta\leq lh(\bar{U})$, in $V^{R_{\bar{U}\restriction \eta}}$, $\langle f_{x^\theta_k}\restriction C_{G_\eta}: k<lh(\bar{U})^{+}\rangle$ is cofinal in $(\Pi_{\alpha\in C_{G_\eta}} g(\alpha), <^*)$.
\end{claim}

\begin{proof}
Given $p=p_0\fr (\bar{U}, A_p)\in R_{\bar{U}\restriction \eta}$ and $\dot{t}\in \Pi_{\alpha\in C_{G_\eta}} g(\alpha)$, we may directly extend $p$ if necessary and assume that for any $\bar{w}\in A_p$, there exists a $R_{\bar{w}}$-name $h(\bar{w})$ such that $p\fr \bar{w}\Vdash \dot{t}(\theta(\bar{w}))=h({\bar{w}})$. For any $\eta'<\eta$, $j(h)(\bar{U}\restriction \eta')$ is an $R_{\bar{U}\restriction \eta'}$-name for an ordinal in $j(g)(\theta)=lh(\bar{U})^+$. By the fact that $R_{\bar{U}\restriction \eta'}$ is $\theta^+$-c.c for any $\eta'<\eta$, we can find some $k<lh(\bar{U})^+$ large enough such that for any $\eta'<\eta$, $\Vdash_{R_{\bar{U}\restriction \eta'}} j(h)(\bar{U}\restriction \eta')<k=j(f_{x^\theta_k})(\theta)$. As a result, $\{\bar{w}\in A_p: p\fr \bar{w}\Vdash \dot{t}(\theta(\bar{w}))<f_{x^\theta_k}(\theta(\bar{w}))\}\in \bigcap\bar{U}\restriction \eta$. The conclusion follows from genericity.
\end{proof}

Suppose we are given $\langle h_\gamma\in \Pi_{\alpha\in C_G\cap \gamma} g(\alpha): \gamma\in \lim C_G\rangle$, we need to find some $f\in \Pi_{\alpha\in C_G} g(\alpha)$ such that on a tail $\gamma\in \lim C_G$, $h_\gamma <^* f\restriction \gamma$. 

We proceed with the density argument. Given $p=p_0\fr (\bar{U}, A)\in R_{\bar{U}}$, by taking a direct extension if necessary, we may assume that for any $\bar{w}\in A$, $p\fr \bar{w}\Vdash \dot{h}_{\theta(\bar{w})}=\dot{h}_{\bar{w}}$, where $\dot{h}_{\bar{w}}$ is a $R_{\bar{w}}$-name. As a result, for any $\eta < lh(\bar{U})$, there is an $R_{\bar{U}\restriction \eta}$-name $\dot{h}_{\bar{U}\restriction \eta}$ such that $j(p)\fr \bar{U}\restriction \eta$ forces that $j(\dot{h})_\theta = \dot{h}_{\bar{U}\restriction \eta}$. By Claim \ref{claim: cofinal} and the $\theta^{+}$-c.c of $R_{\bar{U}\restriction \eta}$, there exists $i_\eta<lh(\bar{U})^{+}$ such that $\Vdash_{R_{\bar{U}\restriction \eta}} \dot{h}_{\bar{U}\restriction \eta}<^* f_{x^\theta_{i_{\eta}}}\restriction C_{G_\eta}$. Let $i=\sup_{\eta<lh(\bar{U})} i_\eta+1$, by Claim \ref{claim: increasing} we know that $j(p)\fr \bar{U}\restriction \eta \Vdash_{j(\vec{R}_{\bar{U}})} j(\dot{h})_\theta<^* f_{x^\theta_i}\restriction C_{G_\eta}$.
As a result, $A'=\{\bar{w}\in A: p\fr \bar{w}\Vdash \dot{h}_{\theta(\bar{w})}<^* f_{x^\theta_i}\restriction C_G\cap \theta(\bar{w})\}\in \bigcap \bar{U}$.
Note that $j(f_{x^\theta_i})\restriction \theta = f_{x^\theta_i}$ since $j(x^\theta_i)\cap \theta=x^\theta_i$. As a result, the extension $p_0\fr (\bar{U}, A')$ of $p$ forces the desired conclusion. \end{proof}

\begin{proposition}\label{proposition: RadinDefinableWO}
Suppose $lh(\bar{U})\in \mathrm{cof}(>\theta)$. Then in $V^{R_{\bar{U}}}$, $\mathcal{F}=_{\mathrm{def}} (P(\theta))^V$
\begin{enumerate}
    \item is closed under restrictions and branching and
    \item admits a $\theta$-upward uniformly definable well-order $<_\theta$.
\end{enumerate}

\end{proposition}

\begin{proof}
Let $G\subset R_{\bar{U}}$, and work in $V[G]$. Let $\mathcal{F}=V\cap P(\theta)$. Note that $\mathcal{F}$ is closed under restrictions and branching. The latter is due to the theorem by Cummings and Woodin that forcing with $R_{\bar{U}}$ does not add fresh subsets to $\theta$.
For $X,Y\in \mathcal{F}$, $X<_\theta Y$ iff for a tail of $\alpha\in C_G$, $f_X(\alpha)<f_Y(\alpha)$. Let us verify that $<_\theta$ is indeed a $\theta$-upward uniformly definable well-order.
\begin{enumerate}
\item It can be readily checked that $<_\theta$ is definable in $H(\theta^+)$ using parameters $C_G$ and $\vec{x}=\langle x^\alpha_i: i<2^\alpha, \alpha<\theta\rangle$.
\item We check that $<_\theta$ is a well-order. It suffices to check that for any $i_0,i_1<2^\theta$, $x^\theta_{i_0}<_\theta x^\theta_{i_1}$ iff $i_0<i_1$. This follows from Claim \ref{claim: increasing}.
\item Let $N$ be a $\theta$-model containing $C_
G$ and $\vec{x}$. Then for any $X, Y\in N\cap V$, $X<_\theta Y$ iff for a tail $\alpha\in C_G$, $f_X(\alpha)<f_Y(\alpha)$, which is absolute between $N$ and $V[G]$.
\end{enumerate}\end{proof}

\begin{remark}
By Proposition \ref{proposition: RadinDefinableWO}, Theorem \ref{theorem:WCFullDiamond} and a theorem of Magidor asserting that in any Radin forcing extension where $\theta$ remains regular $V$-$\diamondsuit(\theta)$ is equivalent to $\diamondsuit(\theta)$, we can conclude that in any Radin forcing extension, if $\theta$ is weakly compact then $\diamondsuit(\theta)$ holds. This was originally proved in \cite{JingOmerRadin}, where the proof there is different and contains some other information. For example, it contains correspondences between certain large cardinal properties in the Radin model and measure sequences satisfying various properties in the ground model.
\end{remark}

\subsection{Separating $\diamondsuit_{\vec{\mathrm{NS}}, \vec{\beth}}$ from $\diamondsuit$}

In this subsection, we show that without the GCH assumption, the conclusion of Proposition \ref{proposition:gch} may not hold.
\begin{theorem}\label{theorem: SeparatingDiamonds}
Suppose there is a strong cardinal $\theta$ in $V$. Then in some forcing extension of $V$, we have $\diamondsuit_{\vec{\mathrm{NS}}, \vec{\beth}}(Sing^\theta_{>\omega})$ holds but $\diamondsuit(\theta)$ fails.
\end{theorem}

\begin{proof}
 By some preliminary forcing, we may suppose in $V$, there is a strong cardinal $\theta$ with $2^\theta=\lambda>cf(\lambda)=\theta^+$. Let $\bar{U}$ be a measure sequence of length $\theta^+$ derived from an embedding $j: V\to M$. We show that in $V^{R_{\bar{U}}}$, $\diamondsuit_{\vec{\mathrm{NS}}, \vec{\beth}}(Sing^\theta_{>\omega})$ holds. In fact, we show something stronger. Note that by a theorem of Woodin, $\diamondsuit(\theta)$ fails in this model.

For each $\alpha<\theta$, let $\langle x^\alpha_j: j<2^\alpha\rangle$ enumerate the subsets of $\alpha$. Let $l_\alpha: cf(2^\alpha)\to 2^\alpha$ be a cofinal map. By the hypothesis, we know a $\bigcap\bar{U}$-measure one set of $\alpha$ satisfies $cf(2^\alpha)=\alpha^+$. For each measure sequence $\bar{w}$, define $t_{\bar{w}}$ to be a function on $\MS \cap V_{\theta(\bar{w})}$ such that $\bar{v}$ is mapped to the index of $x^{\theta(\bar{w})}_{l_{\theta(\bar{w})}(lh(\bar{w}))}\cap \theta(\bar{v})$. We argue that in $V^{R_{\bar{U}}}$, the sequence $\langle t_{\bar{w}}\restriction \MS_G: \bar{w}\in G, \omega<cf^{V[G]}(\theta(\bar{w}))<\theta(\bar{w})\rangle$ is a $\diamondsuit_{\vec{\mathrm{NS}}, \vec{\beth}}(Sing^\theta_{>\omega})$-sequence.

Given $p\in R_{\bar{U}}$ and $\dot{g}\in \Pi_{\alpha\in C_G} 2^\alpha$, by directly extending $p$ if necessary, we may assume that for any $\bar{w}\in A_p$,
\begin{itemize}
    \item $cf(2^{\theta(\bar{w})})=\theta(\bar{w})^+$,
    \item $p\fr \bar{w}\Vdash \dot{g}(\theta(\bar{w}))=h(\bar{w})$ for some $R_{\bar{w}}$-name $h(\bar{w})$ for an element in $2^{\theta(\bar{w})}$.
\end{itemize}
 Since $R_{\bar{w}}$ satisfies $\theta(\bar{w})^+$-c.c, we may find a function $m$ with domain $A_p$ such that for each $\bar{w}\in A_p$, $m(\bar{w})\in \theta(\bar{w})^+$ such that $p\fr \bar{w}\Vdash \dot{g}(\theta(\bar{w}))<l_{\theta(\bar{w})}(m(\bar{w}))$. It suffices to prove the following claim.
\begin{claim}
The collection $B=\{\bar{w}: \{\bar{v}\in V_{\theta(\bar{w})}: p\fr \bar{w} \fr \bar{v}\Vdash \dot{g}(\theta(\bar{v}))<t_{\bar{w}}(\bar{v})\}\in \bigcap \bar{w}\}$ is $\bar{U}$-positive.
\end{claim}
\begin{proof}
Let $f: \theta^+\to \theta^+$ be defined such that $f(\xi)=j(m)(\bar{U}\restriction \xi)$. Let $C$ be the collection closure points of $f$. Namely, $\tau\in C$ iff for any $\gamma<\tau$, $f(\gamma)<\tau$. Fix $\tau\in C$ with $cf(\tau)\in (\omega, \theta)$. We want to show that $B\in U(\tau)$. It suffices to show $\{\bar{v}\in V_\theta: j(p)\fr \bar{U}\restriction \tau \fr \bar{v}\Vdash j(\dot{g})(\theta(\bar{v}))<j(t)_{\bar{U}\restriction \tau}(\bar{v})\}\in \bigcap \bar{U}\restriction \tau$. By the definition, it suffices to show that for any $\xi<\tau$, $$j(j(p))\fr j(\bar{U}\restriction \tau)\fr \bar{U}\restriction \xi \Vdash j(j(\dot{g}))(\theta)<j(j(t)_{\bar{U}\restriction \tau})(\bar{U}\restriction \xi).$$
Fix $\xi<\tau$. By elementarity, 
\begin{itemize}
    \item $j(j(p))\fr \bar{U}\restriction \xi\Vdash j(j(\dot{g}))(\theta)<j(j(l))_\theta(j(j(m))(\bar{U}\restriction \xi))=j(l)_\theta(j(m)(\bar{U}\restriction\xi))=j(l)_\theta(f(\xi))$, and
    \item $j(p)\fr \bar{U}\restriction \tau\Vdash \forall \alpha<\theta, j(t)_{\bar{U}\restriction \tau}(\alpha)=index_\alpha(x^\theta_i \cap \alpha)$ where $i=j(l)_\theta(\tau)$.
\end{itemize}
As a result, $j(j(p))\fr j(\bar{U}\restriction \tau)\fr \bar{U}\restriction \xi$ forces that $$j(j(\dot{g}))(\theta)<j(l)_\theta(f(\xi))<j(l)_\theta(\tau)=index_\theta(j(x^\theta_i)\cap \theta)=j(j(t)_{\bar{U}\restriction \tau})(\bar{U}\restriction \xi),$$ as desired (we use the fact that $\tau$ is a closure point for $f$ in the second inequality). 
\end{proof}\end{proof}

\section{Overlapping Extenders Forcing of Inaccessible length}\label{Section: forcing}

We give a brief account of Gitik's overlapping extenders forcing from \cite{GITIK:OverlappingEBF} with several minor changes in terminology and notations, which we adopt for ease of presentation of later sections. 
We follow the Jerusalem forcing convention of \cite{GITIK:OverlappingEBF}, by which a condition $p$ is stronger (more informative) than $q$ is denoted by $p \geq q$. 

For all relevant notations involving measures, extenders, and Prikry-type forcings, we 
refer the reader to \cite{Gitik-HB}.

The forcing $\po_{\vec{E}}$ is based on a  sequence of overlapping extenders $\vec{E} = \la E_\alpha \mid \alpha < \theta\ra$ in $V$ that satisfy several simple conditions. We list these conditions below and refer to a sequence $\vec{E}$ which satisfies them as a \emph{forcing suitable} sequence. 

\begin{definition}(Forcing suitable extender sequence)\\
Let $\vec{E} = \la E_\alpha : \alpha < \theta\ra$ be a sequence of extenders 
whose critical points are $\kappa_\alpha = \cp(E_\alpha)$.
The ultrapower embedding of $E_\alpha$ is denoted by $j_{E_\alpha} : V \to M_{E_\alpha}$. 

We say that $\vec{E}$ is \emph{forcing suitable} if
it has the following properties:
\begin{enumerate}
    \item The sequence of critical points $\la \kappa_\alpha \mid \alpha < \theta\ra$ is strictly increasing and discrete. 
    Namely, for every $\beta < \theta$, $\kappa_\beta > \vec{\kappa}_\beta := \cup_{\alpha < \beta}\kappa_\alpha$.
    
    \item There exists some $\lambda \geq \cup_{\alpha < \theta}\kappa_\alpha$ such that for each $\alpha < \theta$, $E_\alpha$ is a $(\kappa_\alpha,\lambda)$-extender consisting of a directed system of $\kappa_\alpha$ complete measures $E_\alpha = \{ E_\alpha(a), \pi_{a,b} : a \subseteq b \in [\lambda]^{<\omega}\}$
    whose direct limit embedding $j_{E_\alpha} : V \to M_{E_\alpha}$ satisfies $H_\lambda \subseteq M_{E_\alpha}$.

    \item $E_\alpha \mo E_\beta$ for all $\alpha < \beta$. Moreover, there exists a function $e^\beta_\alpha : \kappa_\beta \to V_\beta$ such that $E_\alpha = j_{E_\beta}(e^\beta_\alpha)(\kappa_\beta)$. 
\end{enumerate}
\end{definition}

Fix a sequence of functions $\vec{e} = \la e^\beta_\alpha \mid \alpha < \beta < \theta\ra$ as above. 
It follows that for each $\beta < \theta$, the function $\lambda^\beta : \kappa_\beta \to \kappa_\beta$ which maps $\rho < \kappa_\beta$ to the length of the extender $e^\beta_0(\rho)$ (i.e., $e^\beta_0(\rho)$ is a $(\kappa_0,\lambda^\beta(\rho))$-extender) represents $\lambda = j_{E_\beta}(\lambda^\beta)(\kappa_\beta)$ in $M_{E_\beta}$.

We denote $\lambda$ by $\lambda(\vec{E})$, and for $\beta < \theta$, denote the restricted sequences $\la E_\gamma \mid \gamma \geq \beta\ra$ by $\vec{E}\setminus \beta$, and 
$\la E_\alpha \mid \alpha < \beta\ra$ by $\vec{E}\uhr \beta$.

We start with an auxiliary poset $\po^*_{\vec{E}}$.

\begin{definition}${}$
\begin{enumerate}
    \item The poset $\po^*_E$ associated to a $(\kappa,\lambda)$-extender $E$ is the Cohen-type poset, consisting of partial functions $f : \lambda \to \kappa$ of size $|f| \leq \kappa$, and further satisfy that 
    $f(\kappa) < \kappa$ is an inaccessible cardinal. The order of $\po^*_{E}$ is containment.
    
    We will occasionally write $\add(\kappa^+,\lambda)$ to denote this forcing. 
    
    \item For a sequence of extenders $\vec{E} = \la E_\alpha \mid\alpha < \theta\ra$ the poset $\po^*_{\vec{E}} = \prod_{\alpha < \theta} \po^*_{E_\alpha}$, with the usual full-support product ordering.
    
    It will also be useful to consider the tail-ordering $\leq_{\bd}$ on $\po^*_{\vec{E}}$ by which a sequence $\vec{f} = \la f_\alpha \mid \alpha < \theta\ra$ of $f_\alpha  \in \po^*_{E_\alpha}$ is $\leq_{\bd}$-stronger than a sequence $\vec{g} = \la g_\alpha \mid \alpha < \theta\ra$ if 
    $f_\alpha \geq_{\po^*_{E_\alpha}} g_\alpha$ for all but boundedly many $\alpha <  \theta$.

    \item 

    Let $\vec{f} = \la f_\alpha \mid \alpha < \theta\ra$ be a sequence of functions. For each $\beta < \theta$ for which $\kappa_\beta \in \dom(f_\beta)$, define
    
    $$\rho^{\vec{f}}_\beta = f_\beta(\kappa_\beta),$$
        and 
        $$\lambda^{\vec{f}}_\beta = \lambda^{\beta}(\rho^{\vec{f}}_\beta).$$

    \item 
Let  $\vec{f} = \la f_\alpha \mid \alpha < \theta\ra$ be a sequence of functions and $s \in [\theta]^{<\omega}$ such that $\kappa_\beta \in \dom(f_\beta)$ for each $\beta \in s$. 

Define the sequence of extender $\vec{E}^{\vec{f},s} = \la E^{\vec{f},s}_\alpha \mid \alpha < \theta\ra$ and sequence of lengths $\vec{\lambda}^{\vec{f},s}$ as follows:
     \begin{itemize}
        \item $E^{\vec{f},s}_\alpha = E_\alpha$ and $\lambda^{\vec{f},s}_\alpha = \lambda(\vec{E})$ for every $\alpha \geq \max(s)$, 
        \item for $\alpha < \max(s)$, if  $\beta = \min(s \setminus (\alpha+1))$ then 
        $E^{\vec{f},s}_\alpha = e^\beta_{\alpha}(\rho^{\vec{f}}_\beta)$ and 
        $\lambda^{\vec{f},s} = \lambda_\beta^{\vec{f}}$. 
        
    \end{itemize}

    \item
    We say that a sequence of functions $\vec{f} =\la f_\alpha \mid \alpha < \theta\ra$ is 
    $(\vec{E},s)$-compatible if it satisfies the following requirements:
    \begin{itemize}
        \item $\kappa_\beta \in \dom(f_\beta)$ for every $\beta \in s$, 
        
        \item for every $\alpha < \theta$, 
        $E^{\vec{f},s}_\alpha$ is a $(\kappa_\alpha,\lambda^{\vec{f},s}_\alpha)$-extender, and 
        $f_\alpha \in \po^*_{E^{\vec{f},s}_\alpha}$. 
    \end{itemize}

\end{enumerate}
\end{definition}

To move forward and define the domain of the poset $\po_{\vec{E}}$, one needs to introduce the measures $E_\alpha(f_\alpha)$ for $f_\alpha \in \po^*_{E_\alpha}$. 

\begin{definition}

Let $E$ be a $(\kappa,\lambda)$-extender, with an associate ultrapower embedding $j_E : V \to M_E$. 

\begin{enumerate}
    \item 
    For a subset $d \in [\lambda]^{\leq\kappa}$ with $\kappa \in d$, define:
    \begin{itemize}
        \item $\mc{E}{d}= \{ (j_E(\tau),\tau) \mid \tau \in d\}$, and 
        \item $E(d) = \{ X \subseteq [\lambda \times \kappa]^{<\kappa} \mid \mc{E}{d} \in j_E(X)\}$.
        
        \item The set of ``typical'' $E(d)$ objects, denoted $OB_E(d)$ is the set of all  partial functions $\nu : d \to \kappa$ with the following properties:
        \begin{enumerate}
            \item $\kappa\in \dom(\nu)$, 
            \item $|\nu| = \nu(\kappa)$, 
            \item $\nu$ is order-preserving.
        \end{enumerate}
    \end{itemize}

    \item 
    For a function $f \in \po^*_E$ with $\kappa \in \dom(f)$, define $E(f) = E(\dom(f))$ and $OB_E(f) = OB_E(\dom(f))$. 
\end{enumerate}
\end{definition}

\begin{definition}(The domain of $\po_{\vec{E}}$)\\

As a set, $\po_{\vec{E}} = \biguplus_{s \in [\theta]^{<\omega}} \po^s_{\vec{E}}$ is union of disjoint parts $\po^s_{\vec{E}}$, $s \in [\theta]^{<\omega}$.
For each $s \in [\theta]^{<\omega}$, $\po^s_{\vec{E}}$ consists of sequences $p = \la p_\alpha \mid \alpha < \theta\ra$, for which 
there exists a $(\vec{E},s)$-compatible sequence of functions $\vec{f} = \la f_\alpha \mid \alpha < \theta\ra$, and the following properties hold:
\begin{enumerate}

    \item For each $\alpha \in s$, $p_\alpha = \la f_\alpha\ra$,
    
    \item for each $\alpha \in \theta\setminus s$, 
    $p_\alpha = \la f_\alpha,A_\alpha\ra$ where
    $A_\alpha \in E^{\vec{f},s}_\alpha(f_\alpha)$ is contained in $OB_{E^{\vec{f},s}_\alpha}(f_\alpha)$, and
   \item for any $\alpha<\alpha'\in \theta$ and any $\nu\in A_{\alpha'}$, $\dom(f_\alpha)\subset \dom(\nu)$. In particular, $\langle \dom(f_\alpha): \alpha<\theta\rangle$ is increasing.
\end{enumerate}

For a condition $p \in \po_{\vec{E}}$ as above, we define
\begin{itemize}
    \item The support of $p$, denoted $s(p)$, is the unique $s \in [\theta]^{<\omega}$ such that
    $p \in \po^s_{\vec{E}}$.
    
    \item $\vec{f} = \vec{f}^p$ and denote each $f_\alpha$ by $f^p_\alpha$, and $A_\alpha$ by $A^p_\alpha$. 
    
    \item $\vec{E}^p =\vec{E}^{\vec{f},s}$ and $E^p_\alpha = E^{\vec{f},s}_\alpha$ for each $\alpha <\theta$.
    
    \item $\vec{\lambda}^p = \vec{\lambda}^{\vec{f},s}$ and $\lambda^p_\alpha = \lambda^{\vec{f},s}_\alpha$ for every $\alpha < \theta$. 
    \end{itemize}
\end{definition}

\begin{definition}\label{Def:ProjOperatorOnSets&Funcs}(The projection operator $\pi_\nu^*$ on functions and sets)\\

\begin{enumerate}
    \item 
Let $\nu \in OB_{E'}(f')$ for some extender $E'$ which is a $(\kappa',\lambda)$-extender and $f' \in \po^*_{E'}$ with $\kappa' \in \dom(\nu)$.

We define the projection operator $\pi^*_\nu$ on suitable functions $g$, sets $A$, and extenders $E \mo E'$ with a designated representing function $e$. 

\begin{enumerate}
    \item For a function $g : \lambda \to \kappa$ for some $\kappa$, so that $\dom(g) \subseteq \dom(\nu)$, we define $\pi^*_\nu g$ to be the function $g'$ given by  $\dom(g') = \{ \nu(\tau) \mid \tau \in \dom(g)\}$ and $g'(\nu(\tau)) = g(\tau)$ for every $\tau \in \dom(g)$.
    
    \item For a set $A$ consisting of functions $\mu \in OB_{E}(g)$ for some extender $E$ with critical point $\kappa < \nu(\kappa')$, and a function $g$ with $\dom(g) \subseteq \dom(\nu)$, 
    (which implies that $\dom(\mu) \subseteq \dom(\nu)$ for all $\mu \in A$) we define
    \[ \pi^*_\nu A = \{ \pi^*_\nu \mu \mid \mu \in A\}.\]
    
    \item Suppose that $E \mo E'$, with $E = j_{E'}(e)(\kappa')$. Define
    $\pi^*_{\nu}(E) = e(\nu(\kappa'))$.
    \end{enumerate}

    \item Suppose that $\vec{E} = \la E_\alpha \mid \alpha < \beta\ra$ is a sequence of length $\beta < \kappa'$ with critical points $\kappa_\alpha < \kappa'$, so that $\vec{E} \in M_{E'}$ is represented by a prescribed sequence of functions $\vec{e} =\la e_\alpha \mid \alpha < \beta\ra$. 
    For every $\nu \in OB_{E'}(g)$ for some $g$  such that $\beta < \nu(\kappa')$, define $$\pi^*_\nu(\vec{E}) = \la \pi^*_\nu(E_\alpha) \mid \alpha < \beta\ra = \la e_\alpha(\nu(\kappa')) \mid \alpha < \beta\ra$$
    and 
    $$
    \pi^*_\nu(\po_{\vec{E}}) = \po_{\pi^*_\nu(\vec{E})}.
    $$
    
    \item 
    Let $\vec{E} = \la E_\alpha \mid \alpha < \beta \rangle$ be a sequence of extenders as above, with a prescribed representing sequence of functions $\vec{e} =\la e_\alpha \mid \alpha < \beta\ra$, and
    $p = \la p_\alpha \mid \alpha < \beta\ra \in \po_{\vec{E}}$
    so that $\dom(f^p_\alpha) \subseteq  \dom(\nu)$ for every $\alpha < \beta$. 
    
    Define $\pi^*_\nu(p)$ to be the sequence $\la \pi^*_\nu(p_\alpha) \mid \alpha < \beta\ra$, such that for every $\alpha < \beta$, 
    $\pi^*_\nu(p_\alpha) = \begin{cases} \la \pi^*_\nu(f^p_\alpha), \pi^*_\nu(A^p_\alpha)\ra & \alpha\not\in s(p) \\
    \la\pi^*_\nu(f^p_\alpha)\ra & \alpha\in s(p)
    
    \end{cases}$, with the above defined projections of functions and measure-one sets of objects. 
\end{enumerate}
\end{definition}

\begin{definition}[One-point extensions and End extensions]\label{definition:one-point}
Given a condition $p \in \po_{\vec{E}}$ and an object $\nu \in A^p_{\alpha}$ for some $\alpha \in \theta \setminus s$  define the one-point extension $p^{+ \nu}$ of $p$ by $\nu$ to be 
given by 

\begin{enumerate}

    \item $p^{+\nu}\uhr \alpha = \pi^*_\nu (p\uhr \alpha)$, 
    
    \item $p^{+\nu}_{\alpha} = (f^p_{\alpha})_{\nu}$, where for any $i\in \dom(f^p_\alpha)=\dom((f^p_\alpha)_\nu)$, $(f^p_{\alpha})_{\nu}(i)=\begin{cases} 
    \nu(i) & i\in \dom(\nu) \\ 
    f^p_\alpha(i) & i\not\in \dom(\nu)
    \end{cases}$.
    
    \item For $\beta > \alpha$, $p^{+\nu}_\beta = \la f^p_\beta, A'_\beta\ra$, where  $A'_\beta = \{ \mu \in A^p_\beta \mid \mu(\kappa_\beta) >  \lambda(\nu(\kappa_\alpha)) \}$
    
    
\end{enumerate}

We say that a condition $p'$ is an end extension of a condition $p$ if it is a condition resulting from a sequence of one-point extensions by objects $\nu_0,\nu_1,\dots, \nu_k$ for some finite $k$. We denote $p'$ by $p^{+\vec{\nu}}$ where $\vec{\nu} = \la \nu_0,\dots,\nu_k\ra$. 
\end{definition}

\begin{definition}[Direct extensions in $\po_{E}$ and the ordering of $\po_{E}$]
Let $p,p^* \in \po_{\vec{E}}$. We say that $p^*$ is a direct extension of $p$ (denoted $p^* \geq p$) if 
$s(p) = s(p^*)$ and for every $\alpha \in \theta \setminus s(p)$, 
 $f^{p}_\alpha \subseteq f^{p^*}_\alpha$ and 
 $$A^{p^*}_\alpha \subseteq \{ \nu \in OB_{E^p_\alpha}(f^{p^*}_\alpha) : \nu \uhr \dom(f^p_\alpha) \in A^p_\alpha\}$$

For conditions $p,q \in \po_{\vec{E}}$, $q$ extends $p$ (namely, $q\geq_{\po_{\vec{E}}} p$) if it is obtained from $p$ by finitely many applications of direct extensions and end extensions.
\end{definition}

\begin{remark}\label{RMK:P_E-basicProperties}
Assuming GCH, $\po_{\vec{E}}$ has the following basic properties: 
\begin{enumerate}
 \item $\po_{\vec{E}}$ has size $|\po_{\vec{E}}| =\lambda$ and satisfies the $\theta^{++}$-c.c. 
    \item for every $p \in \po_{\vec{E}}$ and $\alpha \not\in s(p)$ and $\nu \in A^p_\alpha$,  $\po_{E}/p^{+\nu}$ factors to a product $\po_{\pi_\nu^*(\vec{E}\restriction \alpha)}/(p\uhr \alpha) \times \po_{\vec{E}\backslash \alpha}/(p\downharpoonright \alpha)$, whose first component satisfies $\bar{\kappa}_\alpha^{++}$-c.c, and the direct extension of the second component is $\kappa_\alpha^+$-closed.
\end{enumerate}
\end{remark}

\subsection{Prikry-type properties and capturing dense open sets}\label{Section: Prikry-typeProperties}

\begin{definition}
Let $p \in \po_{\vec{E}}$ be a condition. An $(p,n)$-fat-tree is a tree $T$ of height $n$, with the following properties:
\begin{enumerate}
    \item $stem(T) = \la \ra$ is the empty sequence,
    
    \item for every sequence  $\vec{\nu} =\la \nu_i \mid i < k\ra \in T$, $p^{+\vec{\nu}}$ is an end-extension of $p$
    
    \item  for every sequence  $\vec{\nu} =\la \nu_i \mid i < k\ra \in T$, of non-maximal length $k < n$, there exists some $\alpha_{\vec{\nu}} < \theta$ so that 
    $\suc_T(\vec{\nu}) \in E^p_{\alpha_{\vec{\nu}}}(f^p_{\alpha_{\vec{\nu}}})$,
    
    \item  $\alpha_{\vec{\nu}} < \alpha_{\vec{\nu}\fr\la \mu\ra}$ for every $\vec{\nu} \in T$ and $\mu \in \suc_T(\vec{\nu})$.
    
\end{enumerate}

    We say that a $T$ is fully compatible with $p$ if $\suc_T(\vec{\nu}) = A^p_{\alpha_{\vec{\nu}}}$ for every non-maximal sequence $\vec{\nu} \in T$.
\end{definition}

The following is a simple consequence of the order definition of $\po_{\vec{E}}$.

\begin{lemma}\label{Lem:FatTreeBasics}${}$
\begin{enumerate}
    \item 
If $T$ is a $(p,n)$-fat-tree which is fully compatible with $p$, then the set of end-extensions
$\{ p^{+\vec{\nu}} \mid \vec{\nu} \in T, |\vec{\nu}| = n\}$ 
is predense in $\po_{\vec{E}}/p$.

\item For every condition $p \in \po_{\vec{E}}$ and a $(p,n)$-fat-tree $T$, there exists a direct extension $p^* \geq^* p$ and a $(p^*,n)$-fat subtree $T^*$ of $T$ so that $T^*$ is fully compatible with $p^*$.

\end{enumerate}
\end{lemma}

The following statement is known as the strong Prikry Lemma. We omit the proof as it is similar to the argument in Gitik's singular form from  \cite{GITIK:OverlappingEBF}. 

\begin{proposition}\label{Prop:StrongPrikry}(Strong Prikry Property)\\
Let $D \subseteq \po_{\vec{E}}$ be a dense open set. 
For every condition $p \in \po_{\vec{E}}$ there is a direct extension $p^* \geq^* p$ and a $(p^*,k)$-fat-tree $T$ which is fully compatible with $p^*$ such that for every $\vec{\nu} \in T$ of maximal length $|\vec{\nu}| = k$, $(p^*)^{+\vec{\nu}} \in D$. \end{proposition}

From the strong Prikry Lemma, it is easy to deduce the Prikry property.
 \begin{corollary}\label{corollary:PrikryLemma}
 $\po_{\vec{E}}$ satisfies the Prikry Property, namely, for any condition $p\in \po_{\vec{E}}$ and any statement $\varphi$ in the forcing language, there exists an direct extension $q$ of $p$ such that $q\Vdash \varphi$ or $q\Vdash \neg \varphi$.
 \end{corollary}
 
 Combining Remark \ref{RMK:P_E-basicProperties} and the Prikry property \ref{corollary:PrikryLemma} shows that $\po_{E}$ does not collapse cardinals up to $\theta$ or above $\theta^{+}$
\begin{corollary}
$\po_{\vec{E}}$ preserves all cardinals $\tau \leq \theta$ or $\tau > \theta^{+}$.
\end{corollary}

Our next goal is to verify $\po_{\vec{E}}$ does not collapse $\theta^{+}$ or change the cofinality of cardinals.  
For this we expand the description of the strong Prikry property for meeting a single dense open set $D$, to one which simultaneously deals with $\theta$-many dense open sets.
 
 \begin{definition}
Let $\vec{D}=\la D_\beta \subseteq \po_{\vec{E}}: \beta<\theta\ra$ be a sequence of dense open sets, $p \in \po_{\vec{E}}$, 
We say that $\vec{D}$ is strongly captured by $p$ if for every $\beta > \max(s(p))$ and $\nu \in A^p_\beta$ there exists a dense open subset 
$D_{\nu,<\beta} \subseteq \pi^*_\nu(\po_{\vec{E}^p\uhr \beta})$
such that for every $w \in D_{\nu,<\beta}$, 
there is a
 $(p^{+\nu}_{\geq \beta},n)$-fat tree $T^{\nu,w}$, for some $n < \omega$, such that for every maximal sequence $\vec{\mu} \in T^{\nu,w}$,
\[
w \fr p_{\geq\beta}^{+\la \nu\ra \fr \vec{\mu}} \in D_\beta.
\]
\end{definition}

\begin{lemma}\label{Lem:StrongCapturing}
For every sequence of dense open sets $\vec{D}=\la D_\beta\subseteq \po_{\vec{E}}: \beta<\theta \ra$ and $p \in \po_{\vec{E}}$ there exists a direct extension $p^* \geq^* p$ which strongly captures $\vec{D}$.
\end{lemma}

Before turning to the proof of the Lemma, we introduce a useful terminology of totally $N$-generic conditions.

\begin{definition}\label{Def:TotallyGenericConditions}
Let $(\po^*,\leq^*)$ be a poset which belongs to $H_{\chi}$ for some regular cardinal $\chi$, and $N \elem (H_{\chi},\in,\po^*,\leq^*)$. We say that a condition $f \in \po$ is totally $N$-generic if it bounds an $N$-generic filter for $\po^*$. Namely, the set $\{ f' \in N \cap \po^* : f' \leq f\}$ is generic over $N$.
\end{definition}

\begin{proof}

To simplify the presentation, we assume that $s(p)=\emptyset$. 

We recursively define a $\leq^*$-increasing $\langle q(\beta): \beta\leq\theta\rangle$ such that 
\begin{enumerate}
    \item $q_0=p$,
    \item $q(\beta)\restriction \alpha = q(\alpha)\restriction \alpha$ for any $\alpha<\beta< \theta$.
\end{enumerate}

At limit stages $\delta \leq \theta$, granted the existence of such a sequence, we can define 
$$q(\delta)_{<\delta}=\bigvee_{\alpha<\delta} q(\alpha)_{<\alpha}$$ and  
$$q(\delta)_{\geq \delta} = \bigvee_{\alpha < \delta} q(\alpha)_{\geq \delta}$$

where $\bigvee$ is the natural limit condition obtained by taking coordinate-wise unions of the Cohen functions and intersecting the appropriate pullbacks of the measure one sets. Note that $q(\delta)$ is easily seen to be a condition which satisfies the requirements above.

At stage $\beta+1$, given $q(\beta)$, we take  an elementary submodel $N\prec H(\chi)$  containing relevant objects including $\po_{\vec{E}},D=D_\beta$, and $q(\beta)$ with $|N|=\kappa_\beta$ and $^{<\kappa_\beta} N\subset N$. By the $\kappa_\beta^+$-closure of the direct extension ordering of $\po^*_{E^p(\beta)}\times \po^*_{\vec{E}\setminus \beta}$, we can find a totally $N$-generic condition $f^*_\beta \fr r^*_{>\beta}$ with $\dom(f^*_\beta)=N\cap \lambda(\vec{E}^p)$. Let $$A_\beta^* = \{\nu\in OB_{E_\beta}(f^*_\beta)\mid \nu\restriction \dom(f^{q(\beta)}_\beta)\in A^{q(\beta)}_\beta\},$$ 

and define 

$$q(\beta+1) = q(\beta)\uhr\beta \fr \la f^*_\beta,A^*_\beta\ra \fr r^*_{>\beta}.$$

For each $\nu \in A^*_\beta$, the forcing above $q(\beta+1)^{+\nu}$ breaks into a product 
$\pi^*_\nu(\po_{\vec{E}^p\uhr\beta}) \times \po_{\vec{E}^p\setminus \beta}$.
For each $w \in \pi^*_\nu(\po_{\vec{E}^p\uhr\beta})$, let $D_{w,\geq \beta} \subseteq \po_{\vec{E}^p\setminus \beta}$  be the set of all condition $q \geq q(\beta)^{+\nu}_{\geq \beta}$ such that either 
\begin{itemize}
    \item $q$ forces that there is no $q' \in \name{G}(\po_{\vec{E}^p\setminus \beta})$ for which $w \fr q' \in D$, or
    \item $w \fr q \in D$.
\end{itemize}
$D_{w,\geq \beta}$ is clearly dense open in $\po_{\vec{E}^p\setminus \beta}$
and belongs to $N$. 

By the strong Prikry property, there exists
a dense open set 
$D^*_{w,\geq \beta}
\subseteq \po^*_{E_\beta^p} \times \po_{\vec{E}^p\setminus(\beta+1)}$ in the direct extension ordering, such that for every $q^\nu_{\geq\beta} \in D^*_{w,\geq \beta}$ there is a  $(q^\nu_{\geq\beta},n_{\nu,w})$-fat-tree $T^{\nu,w}$, such that for every maximal sequence $\vec{\mu} \in T^{\nu,w}$, 
$(q^\nu_{\geq\beta})^{+ \vec{\mu}} \in D_{w,\geq \beta}$.

Since $\nu\in N$, $q(\beta)^{+\nu}, D_{w,\geq \beta}$, and $D^*_{w,\geq \beta}$ belong to $N$ as well. It follows from our $N$-genericity assumption of $q(\beta+1)$, that 
$q(\beta+1)^{+\nu}_{\geq\beta} \in D^*_{w,\geq \beta}$. 

Let 
\[D_{\nu,<\beta} = \{w \in \pi^*_\nu(\po_{\vec{E}^p\uhr\beta})
\mid \exists q'  \geq  q(\beta+1)^{+\nu}_{\geq \beta}. \ w \fr q' \in D.\} \]

$D_{\nu,<\beta}$ is clearly dense open in 
$\pi^*_\nu(\po_{\vec{E}^p\uhr\beta})$. Fix $w \in D_{\nu,<\beta}$. 
We verify that for every maximal sequence in the associated $(q(\beta+1)^{+\nu}_{\geq \beta},n_{\nu,w})$-fat-tree $T^{\nu,w}$, $\vec{\mu}$, the condition 
$w \fr q(\beta+1)^{+\la \nu\ra \fr \vec{\mu}}_{\geq \beta}$ belongs to $D$.

By the definition of $D_{w,\geq\beta}$, it suffices to verify that the condition $q(\beta+1)^{+\la \nu\ra \fr \vec{\mu}}_{\geq \beta}$ forces 
the  statement $\sigma_w$
 in forcing language of $\po_{\vec{E}^p\setminus\beta}$, 
\[\sigma_w := \exists q' \in \name{G}. \ w \fr q' \in D.\] 

As a matter of fact, we show that the weaker condition
$q(\beta+1)^{+\la \nu\ra}$ forces $\sigma_w$.
To see this, recall that by the Prikry Lemma applied to $\po^*_{E_\beta^p} \times \po_{\vec{E}^p\setminus(\beta+1)}$, the set of direct extensions of $q(\beta)^{+\nu\uhr \dom(f^q_\beta)}_{\geq \beta}$ deciding $\sigma_w$ is dense open above $q(\beta)^{+\nu\uhr \dom(f^q_\beta)}_{\geq \beta}$ in the direct extension ordering, and belongs to $N$.
It follows that $q(\beta+1)^{+\nu}_{\geq \beta}$ decides $\sigma_w$, and
since $w \in D_{\nu,<\beta}$ it must be that 
\[
q(\beta+1)^{+\nu}_{\geq \beta} \Vdash \sigma_w.
\]
\end{proof}

\begin{proposition}\label{proposition:almostSacks}
For any $p\in \po_{\vec{E}}$ and for any name $\dot{f}: \theta \to \Ord$, there exists a direct extension $q\geq^* p$ and a function $g: \theta\to [\Ord]^{<\theta}$ such that $q\Vdash \dot{f}(\alpha)\in g(\alpha)$ for all $\alpha<\theta$.
\end{proposition}

\begin{proof}
We may without loss of generality assume $s(p)=\emptyset$. Since otherwise, letting $\gamma=\max s(p)$, then we can decompose $\po_{\vec{E}}/p$ into $\po_0\times \po_1=\po_{\vec{E}^p\restriction \gamma+1}/p\restriction \gamma+1 \times \po_{\vec{E}\backslash \gamma+1}/p_{>\gamma}$. Since $\po_0$ has $\kappa_\gamma^{+}$-c.c, we can work in $\po_1$ above $p_{>\gamma}$ where the support is empty to obtain the desired conclusion.

Let $D_\beta=\{r\in \po_{\vec{E}}: \exists \delta \ r\Vdash \dot{f}(\beta)=\delta\}$ for $\beta<\theta$. It is easy to check that $D_\beta$ is a dense open set. Apply Lemma \ref{Lem:StrongCapturing} to $p$ and $\vec{D}=\la D_\beta: \beta<\theta\ra$ to get 
\begin{itemize}
    \item $q\geq^* p$ and 
    \item for any $\nu \in A^{q^*}_\beta$, a dense open subset $D_{\nu,<\beta}$ of $\pi_{\nu}^*(\po_{\vec{E}\restriction \beta})$ such that for any $w\in D_{\nu,<\beta}$, there exists a  $(q^{+\nu}_{\geq \beta}, n^{\nu,w})$-fat tree $T^{\nu,w}$ such that for any maximal sequence $\vec{\mu}\in T^{\nu, w}$, $(w\fr q)^{+\la \nu \ra \fr \vec{\mu}}\in D_\beta$, as witnessed by $\delta^{\nu, w, \vec{\mu}}\in \Ord$.
\end{itemize}
Define $g(\beta)=\bigcup\{\delta^{\nu, w, \vec{\mu}}: \nu\in A^{q}_\beta, w\in D_{\nu,<\beta}, \vec{\mu} $ is a maximal sequence from $T^{\nu, w}\}$. It is easy to see that $q$ and $g$ are as sought.
\end{proof}

\begin{corollary}\label{cor:nofastclub}
${ }$
\begin{enumerate}
    \item The cofinalities and cardinalities of $\theta$ and $\theta^+$ are preserved, 
    
    \item $\po_{\vec{E}}$ preserves all cofinalities and cardinalities, and 
    
    \item For every $\po_{\vec{E}}$-name of a closed unbounded subsets $\name{D}$ of $\theta$ and a condition $p$, there is a direct extensions $q$ of $p$ and a closed unbounded set $C \subseteq \theta$ such that $q \Vdash \check{C} \subseteq \name{D}$.
\end{enumerate}
\end{corollary}

\begin{proof}${}$
\begin{enumerate}
    \item Let $p\in \po_{\vec{E}}$ and $\dot{f}: \rho\to \theta$, $\rho<\theta$ be given. By Proposition \ref{proposition:almostSacks}, there is $q\geq p$ and $g: \rho\to [\theta]^{<\theta}$ such that $q\Vdash \dot{f}(\alpha)\in g(\alpha)$ for any $\alpha<\rho$. Then $q\Vdash ran(\dot{f})\subset \sup \bigcup \{g(\alpha): \alpha<\rho\}<\theta$ since $\theta$ is regular. Similar argument shows that the regularity of $\theta^+$ is preserved.
    
    \item This is an immediate consequence of previous item (1)  and the factoring property of $\po_{\vec{E}}$ from Remark \ref{RMK:P_E-basicProperties}.
    
    \item Given $p\in \po_{\vec{E}}$ and $\dot{D}$ a $\po_{\vec{E}}$-name for a club in $\theta$, applying Proposition \ref{proposition:almostSacks}, there exist $q\geq p$ and $g: \theta\to [\theta]^{<\theta}$ such that $q\Vdash \min \dot{D}-\alpha \in g(\alpha)$ for all $\alpha<\theta$. Let $h(\alpha)=\sup g(\alpha)<\theta$. Then $q\Vdash cl_h\subset \dot{D}$ where $cl_h=\{\eta<\theta: \forall \gamma<\eta \ h(\gamma)<\eta\}$ is a club in $\theta$ in $V$.
\end{enumerate}
\end{proof}

\section{The failure of $\diamondsuit_{\vec{\mathrm{NS}}, \vec{\lambda}}(\theta)$ in $V^{\po_{\vec{E}}}$}\label{Section: failure}

\begin{lemma}\label{lemma:NameReduction}
Given a condition $p\in \po_{\vec{E}}$ and a $\po_{\vec{E}}$-name for a sequence $\langle \dot{X}_\alpha\subset \alpha: \alpha<\theta\rangle$, there exists a direct extension $q$ of $p$ and a sequence $\langle X^\alpha_\nu: \nu\in X^q_{\alpha}, \alpha<\theta\rangle$ such that for each $\alpha<\theta$, $\alpha\not\in s(p)$,
\begin{enumerate}
    \item for each $\nu\in A^q_\alpha$, $X^\alpha_\nu$ is a $\pi^*_\nu(\po_{\vec{E}})$-name, 
    \item for each $\nu\in A^q_\alpha$, $q^{+\nu}\Vdash \dot{X}^\alpha_\nu = \dot{X}_\alpha$.
\end{enumerate}

\end{lemma}

\begin{proof}
To simplify the presentation, we may assume that $s(p)=\emptyset$. This is without loss of generality, since we can proceed by induction on the size of $s(p)$. More precisely, suppose $\alpha = \max s(p)$, then $\po_{\vec{E}}/p$ could be broken into $(\po_{\vec{E}^p\uhr \alpha}/p\uhr\alpha) \times (\po_{\vec{E}\setminus \alpha}/p\dhr\alpha)$. Directly extending $p\dhr\alpha$ if necessary, we may assume that $\langle \dot{X}_\beta: \beta<\alpha\rangle$ is a $\po_{\vec{E}^p\uhr \alpha}/p\uhr\alpha$-name. Then the conclusion follows from the induction hypothesis applied to $\po_{\vec{E}^p\uhr\alpha}/p\restriction \alpha $ and the simplified case applied to $\po_{\vec{E}\setminus \alpha}/p\dhr\alpha$.

We recursively define a $\leq^*$-increasing $\langle q_\alpha: \alpha<\theta\rangle$ such that 
\begin{enumerate}
    \item $q_0=p$,
    \item $q_{\beta}\restriction \alpha = q_\alpha\restriction \alpha$ for any $\alpha<\beta< \theta$,
    \item for each $\nu\in A^{q_{\alpha+1}}_{\alpha}$, there exists a $\pi^*_\nu(\po_{\vec{E}\restriction \alpha})$-name $\dot{X}^\alpha_\nu$ such that $q_{\alpha+1}^{+\nu}$ forces that $\dot{X}_\alpha=\dot{X}^\alpha_\nu$.
\end{enumerate}
Granted the existence of such a sequence, we can define $q=\bigcup_{\alpha<\theta} q_\alpha\restriction \alpha$, which is easily seen to be as sought. To construct the sequence, at limit stage $\alpha<\theta$, we just let $q_\alpha=\bigvee_{\beta<\alpha} q_\beta$. At stage $\alpha+1$, let $N\prec H(\chi)$ be an elementary submodel containing relevant objects including $\po_{\vec{E}}$, $q_\alpha$ and $\dot{X}_\alpha$, such that $|N|=\kappa_\alpha$ and $^{<\kappa_\alpha} N\subset N$. By the $\kappa_\alpha^+$-closure of $(\po^*_{E(\alpha)}/f^{q_\alpha}_\alpha \times \po^*_{\vec{E}\backslash \alpha}/\vec{f}^{q_{\alpha}}\downharpoonright \alpha)$, we can find an $N$  totally generic condition $f^*_\alpha \fr \langle r^*_\beta: \beta > \alpha\rangle$ with $\dom(f^*_\alpha)=N\cap \lambda$. Let $A_\alpha^*$ be $\{\nu\in E(f^*_\alpha): \nu\restriction \dom(f^{q_\alpha}_\alpha)\in A^{q_\alpha}_\alpha\}$. Let us show that $q_{\alpha+1}=q_\alpha\restriction \alpha \fr (f^*_\alpha, A^*_\alpha)\fr \langle r^*_\beta: \beta>\alpha\rangle$ is as desired. Let $\nu\in A^*_\alpha$ and $q^*=\pi^*_\nu(q_\alpha\restriction \alpha)$. Notice that since $|\nu|<\kappa_\alpha$, $\nu\subset N$, ${}^{<\kappa_\alpha}N\subset N$ and $q^*\in V_\theta$, we have that $\nu, q^*\in N$. In $N$, consider the set 
$$
D=\{(f,\vec{r})\in (\po^*_{E_\alpha}/f^{q_\alpha}_\alpha) \times (\po^*_{\vec{E}\backslash \alpha}/\vec{f}^{q_{\alpha}}\downharpoonright \alpha) :$$$$ \text{there exists a } \pi^*_\nu(\po_{\vec{E}\restriction \alpha})\text{-name }\dot{B}, \ q^* \fr f_\nu \fr \vec{r}\Vdash \dot{B}=\dot{X}_\alpha\}.$$

Note that $D$ is dense in $\po^*_{E(\alpha)}/f^{q_\alpha}_\alpha \times \po^*_{\vec{E}\backslash \alpha}/\vec{f}^{q_{\alpha}}\downharpoonright \alpha$ and
\begin{itemize}
    \item the number of (nice) $\pi^*_\nu(\po_{\vec{E}\restriction \alpha})$-names for subsets of $\alpha$ is $<\kappa_\alpha$ and
    \item  $(\po^*_{E(\alpha)}/f^{q_\alpha}_\alpha) \times (\po_{\vec{E}\backslash \alpha}/q_{\alpha}\downharpoonright \alpha)$ satisfies the Prikry Property by Corollary \ref{corollary:PrikryLemma} and is $\kappa^+_\alpha$-closed.
\end{itemize}

 As a result, $f^*_\alpha \fr \langle r^*_\beta: \beta > \alpha\rangle$ meets $D$, which in turns implies that $q^*\fr (f^*_\alpha)_\nu\fr \langle r^*_\beta: \beta > \alpha\rangle$ forces that $\dot{X}_\alpha=\dot{X}^\alpha_\nu$ for some  $\pi^*_\nu(\po_{\vec{E}\restriction \alpha})$-name $\dot{X}^\alpha_\nu$.
\end{proof}

\begin{definition}\label{Definition:t-functions}
Let $G\subset \po_{\vec{E}}$ be generic and in $V[G]$, define: 

\begin{enumerate}
\item for $i<\theta$, $\lambda_i=\lambda^i(f^q_{i}(\kappa_i))$ for some (any) $q\in G$ with $i\in s(q)$;
\item for each $\tau<\lambda$, $t_\tau$ is a function on $\theta$ such that for any $\gamma<\theta$, $t_\tau(\gamma)$ is defined and is equal to $\tau_k$ iff there exists some condition $p\in G$ satisfying: 
\begin{itemize}
    \item $\gamma\in s(p)$, 
    \item if we enumerate decreasingly $\{\gamma_0, \cdots, \gamma_{k-1}\}=s(p)-\gamma$ (in particular $\gamma = \gamma_{k-1})$, then $\tau=\tau_0\in \dom(f^p_{\gamma_0})$, $\tau_1=f^p_{\gamma_0}(\tau_0) \in \dom(f^p_{\gamma_1})$, $\tau_2 = f^p_{\gamma_1}(\tau_1) \in \dom(f^p_{\gamma_2})$, $\cdots, \tau_k=f^p_{\gamma_{k-1}}(\tau_{k-1})$.
\end{itemize}
\end{enumerate}

\end{definition}

\begin{proposition}\label{proposition:scale}
Assume that $\lambda\geq \theta^+$.
For any $p\in \po_{\vec{E}}$ and a $\po_{\vec{E}}$-name for a function in $\dot{g}\in \Pi_{i<\theta}\dot{\lambda}_i$, it is true that $p\Vdash \exists \tau<\lambda, \dot{g}<^* \dot{t}_\tau$.
\end{proposition}

\begin{proof}
Given $p\in \po_{\vec{E}}$, and a name for a function $\dot{g}\in \Pi_{i<\theta}\dot{\lambda}_i$, we find a direct extension $q$ of $p$ and some $\tau<\lambda$ such that $q\Vdash \dot{g}<^* \dot{t}_\tau$. We may without loss of generality assume that $s(p)=\emptyset$. The reason is that if $\alpha=\max s(p)$, $\po_{\vec{E}}/p$ can be naturally decomposed at $\alpha$ to $\po_{\vec{E}^p \restriction \alpha}/p\uhr \alpha \times \po_{\vec{E}\backslash \alpha}/p \dhr \alpha$ with the lower part of the decomposition being small. Then we can work with the upper part of the decomposition.

Applying Lemma \ref{lemma:NameReduction} and directly extending $p$ is necessarily, we may assume that for any $\alpha<\theta$ and $\nu\in A^{p}_\alpha$, $p^{+\nu}\Vdash \dot{g}(\alpha)=\dot{g}^\alpha_\nu$ where $\dot{g}^\alpha_\nu$ is a $\pi^*_{\nu}(\po_{\vec{E}\restriction \alpha})$-name for an ordinal below $\lambda^\alpha(\nu(\kappa_\alpha))$.
For each $\nu\in A^p_\alpha$, since $\pi^*_{\nu}(\po_{\vec{E}\restriction \alpha})$ satisfies $\vec{\kappa}_\alpha^{++}$-c.c and $\lambda^\alpha(\nu(\kappa_\alpha))$ is a regular cardinal $>\vec{\kappa}_\alpha^{++}$, there exists an ordinal $\gamma_{\nu}^\alpha<\lambda^\alpha(\nu(\kappa_\alpha))$ such that $\Vdash_{\pi^*_{\nu}(\po_{\vec{E}\restriction \alpha})}  \dot{g}^\alpha_\nu<\gamma_{\nu}^\alpha$. 
For each $\alpha<\theta$, let $\gamma_\alpha=j_\alpha(\nu\mapsto \gamma_{\nu}^\alpha)(\mc{E_\alpha}{f^p_\alpha})<\lambda$. Let $\tau\in \lambda- (\sup_{\alpha<\theta} \gamma_\alpha +1)$. Now we directly extend $p$ to $q$ such that for each $\alpha<\theta$,
\begin{itemize}
    \item $\dom(f^q_\alpha)$ contains $\tau$, and
    \item for any $\nu\in A^q_\alpha$, it holds that $\tau\in \dom(\nu)$ and $\nu(\tau)>\gamma_{\nu}^\alpha$.
\end{itemize}
It is now easy to see that $q$ is as desired, namely, $q\Vdash \dot{g}<^* \dot{t}_\tau$.
\end{proof}

\begin{proposition}\label{proposition:globalScale}
In $V[G]$, $\diamondsuit_{\vec{\mathrm{NS}},\vec{\lambda}}(\theta)$ fails, where $\vec{\lambda}=\langle\lambda_i: i<\theta\rangle$.
\end{proposition}

\begin{proof}
Given $p\in \po_{\vec{E}}$ and a $\po_{\vec{E}}$-name $\langle \dot{g}_\gamma\in \Pi_{i<\gamma}\dot{\lambda}_i: \gamma\in \lim \theta\rangle$, we need to find $q\geq p$ and $\tau<\lambda$ such that $q\Vdash ``\text{for a tail }\gamma\in \lim \theta, \dot{g}_\gamma<^*\dot{t}_\tau\restriction \gamma$''. 
As in the proof of Proposition \ref{proposition:scale}, we may assume $s(p)=\emptyset$. By Lemma \ref{lemma:NameReduction}, directly extending $p$ if necessary, we may assume that for any $\gamma\in \lim \theta$ and $\nu\in A^p_\gamma$, $p^{+\nu}\Vdash \dot{g}_\gamma = \dot{g}^\gamma_\nu$ for some $\pi^*_{\nu}(\po_{\vec{E}\restriction \gamma})$-name $\dot{g}^\gamma_\nu$. 

Fix $\nu\in A^p_\gamma$. By Proposition \ref{proposition:scale} applied to $\pi^*_{\nu}(\po_{\vec{E}\restriction \gamma})$ and $\pi^*_{\nu}(p\restriction \gamma)$, as well as by the fact that $\pi^*_{\nu}(\po_{\vec{E}\restriction \gamma})$ satisfies $\vec{\kappa}_\gamma^{++}$-c.c and $\lambda^\gamma(\nu(\kappa_\gamma))>\vec{\kappa}^+_{\gamma}$ is regular, there exists some $\tau_\nu^\gamma<\lambda^\gamma(\nu(\kappa_\gamma))$ such that $\pi^*_{\nu}(p\restriction \gamma)\Vdash_{\pi^*_{\nu}(\po_{\vec{E}\restriction \gamma})} \dot{g}^\gamma_\nu <^* (\dot{t}_{\tau_\nu})^{\nu}$. We need to clarify one notation here: $(\dot{t}_{\tau_\nu})^{\nu}$ refers to the $t$-functions (see Definition \ref{Definition:t-functions}) defined locally in $V^{\pi^*_\nu(\po_{\vec{E}\restriction \gamma})}$.

For each $\gamma\in \lim \theta$, let $\tau_\gamma = j_\gamma (\nu\mapsto \tau_\nu)(\mc{E_\gamma}{f^p_\gamma})<\lambda$. Let $\tau\in \lambda-(\sup_{\gamma\in \lim \theta} \tau_\gamma+1)$. Directly extend $p$ to $q$ such that for any $\gamma\in \lim \theta$, $\nu\in A^q_\gamma$,
    \begin{itemize}
    \item $\tau\in \dom(\nu)\subset \dom(f^q_\gamma)$,
    \item $\nu(\tau)>\tau_\nu$.
    \end{itemize}
Note that by Definition \ref{Definition:t-functions} applied to $\po_{\vec{E}}$ as well as $\pi^*_{\nu}(\po_{\vec{E}_{<\gamma}})$ and the choice of $q$, we have that: for any $\gamma\in \lim\theta$ and any $\nu\in A^q_\gamma$, $q^{+\nu}\Vdash (\dot{t}_{\nu(\tau)})^{\nu} =^* \dot{t}_\tau\restriction \gamma$. Since $\nu(\tau)>\tau_\nu$, we know that $q^{+\nu} \Vdash \dot{g}_\gamma=\dot{g}^\gamma_\nu<^*(\dot{t}_{\tau_\nu})^{\nu}<^* (\dot{t}_{\nu(\tau)})^{\nu}=^* \dot{t}_\tau\restriction \gamma$ for any $\gamma\in \lim \theta$ and $\nu\in A^q_\gamma$. Therefore, $q$ and $\tau$ are as desired.\end{proof}

With Proposition \ref{proposition:globalScale}, we can complete the proof of Theorem \ref{theorem: mainconsistency}.
\begin{proof}(Theorem \ref{theorem: mainconsistency})\\
Assume without loss of generality that $V\models \mathrm{GCH}$. Let $\la \kappa_\alpha : \alpha < \theta\ra$ be an increasing enumeration of a cofinal discrete set of $\theta^{++}$-strong cardinals $\kappa_\alpha < \theta$. Having each $\kappa_\alpha$ being $\theta^{++}$-strong, with a standard preliminary $\theta$-c.c iteration of Cohen sets that does not changes cofinalities,\footnote{Forcing with an Easton support iteration of Cohen forcings, adding functions $f_\nu : \nu \to V_\nu$ for all $\nu < \theta$,  a standard argument shows that for each $\beta < \theta$ the extender embedding $j_{E_\beta} : V \to M_{E_\beta}$ and each $x \in V_{j_{E_\beta}(\kappa_\beta)}^{M_{E_\beta}}$, the embedding extends to an embedding $j^*$ which further satisfies $j^*(f_{\kappa_\beta})(\kappa_\beta) = x$.}  it is straightforward to recursively define a forcing suitable sequence of extenders $\vec{E} = \la E_\alpha : \alpha < \theta\ra$ such that each $E_\alpha$ is a $(\kappa_\alpha,\theta^+)$-extender. 
Let $G \subseteq \po_{\vec{E}}$ be a generic filter and $\vec{\lambda} = \la \lambda_i : i < \theta\ra$ be the sequence from Definition \ref{Definition:t-functions} of the different reflections of $\lambda = \theta^+$ via points associated with $E_i$, $i < \theta$. Then in $V[G]$, 
\begin{itemize}

    \item 
    By Corollary \ref{cor:nofastclub},
    $V[G]$ and $V$ agree on cardinals and cofinalities, and every club $C \subseteq \theta$ in $V[G]$ contains a club $C'$ in $V$. 
    
    \item for every limit ordinal $\delta \leq \theta$, $\lambda_\delta$ is a regular cardinal, $\vec{\kappa}_\delta^{++} < \lambda_\delta < \kappa_\delta$. 
    We have $\lambda_\delta \geq 2^{\vec{\kappa}_\delta}$ is witnessed by the sequence
     $\vec{t}^\delta \subseteq \prod_{i < \delta} \lambda_i$ generically introduced by the 
    $\vec{E}\uhr \delta$ reflected part of the forcing.
    On the other hand, $2^{\vec{\kappa}_\delta} \leq \lambda_\delta$ since the $E_\delta$-refelcted forcing satisfies $\vec{\kappa}_\delta^{++}$-c.c and has size $\lambda_\delta$. 
    Hence $\lambda_\delta = 2^{\vec{\kappa}_\delta}$ also equals to $2^\delta = \beth(\delta)$ for club many $\delta < \theta$.\\
    
    $\Diamond_{\vec{\mathrm{NS}}, \vec{\lambda}}(\theta)$ fails by Proposition \ref{proposition:globalScale}. Therefore 
    $\Diamond_{\vec{\beth}}(\theta)$ fails as well.  
     
\end{itemize}

\end{proof}

In the last part of this section we show that the assumption  $\lambda\geq \theta^+$ is crucial for the argument. For convenience, for each $\alpha<\theta$, $d\in [\lambda]^{\leq \kappa_\alpha}$ and $A\in E_\alpha(d)$, we may assume that for any $\nu, \nu'\in A$, if $\nu(\kappa_\alpha)=\nu'(\kappa_\alpha)$, then $\dom(\nu)=\dom(\nu')$. This is without loss of generality, since for any $A'\in E(d)$, we can always find $A\subset A'$ in $E(d)$ satisfying the requirement above. This is to ensure that if the value at $\kappa_\alpha$ is determined, then there is a unique $\nu$ responsible for this.

\begin{proposition}
Let $\po_{\vec{E}}$ be defined such that $\lambda=\theta$. If $\diamondsuit(\theta)$ holds in the ground model, then it continues to hold in $V^{\po_{\vec{E}}}$. 
\end{proposition}

\begin{proof}
Assume the ground model $\diamondsuit(\theta)$ gives rise to a sequence $\langle (p_\alpha, f_\alpha): \alpha<\theta\rangle$ such that 
\begin{itemize}
    \item $p_\alpha$ is a condition in $\po_{\vec{E}\restriction \alpha}$ satisfying that for each $\beta<\alpha$, $\dom(f^{p_\alpha}_\beta)\subset \alpha$,
    \item $f_\alpha$ is a function with domain $\bigcup_{\beta<\alpha} A^{p_\alpha}_\beta$ such that for any $\nu\in A^{p_\alpha}_\beta$ where $\beta<\alpha$, $f_\alpha(\nu)$ is a nice $\pi^*_\nu(\po_{\vec{E}\restriction \beta})$-name for a function from $\beta\to 2$.
\end{itemize}

Let $G\subset \po_{\vec{E}}$ be generic over $V$. In $V[G]$, let us define a sequence: at $\alpha$, we need to define a function $g_\alpha$ from $\alpha \to 2$. For $\beta\in \alpha-s(p_\alpha)$ and $\nu\in A^{p_\alpha}_\beta$, we say $\nu$ is \emph{good} if $\pi^*_{\nu}(p_\alpha\restriction \beta)\in G_\beta = \{p\restriction \beta: p\in G, \beta\in s(p)\}$ and $p_\alpha^{+\nu}$ is an initial segment of a condition in $G$. Then we define $g_\alpha$ to be $\bigcup_{\beta, \nu\in A^{p_\alpha}_\beta,\nu \text{ is good}} (f_\alpha(\nu))^{G_\beta}$. We show, using a density argument, that in $V[G]$, $\langle g_\alpha: \alpha<\theta\rangle$ is a $\diamondsuit(\theta)$-sequence.

Suppose we are given $p\in \po_{\vec{E}}$, a $\po_{\vec{E}}$-name $\dot{f}: \theta\to 2$ and a $\po_{\vec{E}}$-name $\dot{D}$ for a club in $\theta$. By Lemma \ref{lemma:NameReduction}, we may assume that $p^{+\nu}$ forces $\dot{f}\restriction \beta =\dot{f}^\beta_\nu$ for every $\nu\in A^p_\beta$, where $\dot{f}^\beta_\nu$ is a nice $\pi_{\nu}^*(\po_{\vec{E}\restriction \beta})$-name. Let $h$ be the function whose domain is $\bigcup_{\beta<\theta} A^p_\beta$ and for $\beta<\theta$ and $\nu\in A^p_\beta$, $h(\nu)=\dot{f}^\beta_\nu$. By Corollary \ref{cor:nofastclub}, we may assume there is a club $D\in V$ such that $p\Vdash D \subset \dot{D}$.

By the property of the $\diamondsuit(\theta)$ in the ground model, we may assume that $E=\{\beta\in D: (p_\alpha, f_\alpha)=(p\restriction \alpha, h\restriction V_{\beta})\}$ is a stationary subset of $\theta$. Fix $\alpha\in E$.
Let $G\subset \po_{\vec{E}}$ be generic with $p\in G$, then it is clear that for any $\beta\in \alpha$, there is a unique $\nu=\nu_\beta\in A^p_\beta$ such that $p^{+\nu}$ is in $G$. As a result, by our choice, $g_\alpha=\bigcup_{\beta<\alpha} (f^\beta_{\nu_\beta})^{G_\beta}=\bigcup_{\beta<\alpha}(\dot{f})^G\restriction \beta=(\dot{f})^G\restriction \alpha$.
\end{proof}

\section{Compactness properties in $V^{\po_{\vec{E}}}$ when the length of $\vec{E}$ is measurable}\label{Section: Compactness}

Suppose $\theta$ is a measurable cardinal carrying a normal measure $U$. We note that usually a weaker large cardinal hypothesis on $\theta$ suffices to run the arguments in this section. We will not optimize in this respect for simplicity.

\begin{definition}
Given a sequence $\vec{r}=\langle r_\alpha: \alpha<\theta\rangle$ where $r_\alpha: \lambda\to \kappa$ is a partial function of size $\leq \kappa$ where $\kappa<\theta<\lambda$. We say $\vec{r}$ forms a \emph{$\Delta$-system with root $f$} if $\{\dom(r_\alpha): \alpha<\theta\}$ forms a $\Delta$-system with root $d^*$ and $f=r_\alpha\restriction d^*$ for any $\alpha<\theta$. We say $\vec{r}$ forms a \emph{$\Delta$-system} if there exists some $f$ such that $\vec{r}$ forms a $\Delta$-system with root $f$.
\end{definition}

\begin{lemma}\label{lemma:singleinstance}
Suppose we are given a sequence $\vec{r}=\langle r_\alpha\in Add(\kappa^+, \lambda): \alpha<\theta\rangle$ where $\kappa<\theta<\lambda$. Then there exists $B\in U$ such that $\{r_\alpha: \alpha\in B\}$ forms a $\Delta$-system.
\end{lemma}

\begin{proof}
Let $m_\alpha= \dom(r_\alpha)$. It suffices to show that there is $B\in U$ such that $\{m_\alpha: \alpha\in B\}$ forms a $\Delta$-system. Since $\bigcup_{\alpha<\theta} m_\alpha\in [\lambda]^{\leq \theta}$, by injecting this set into $\theta$, we may assume without loss of generality that $m_\alpha\in [\theta]^{\leq \kappa}$. Consider the map $\alpha\mapsto m_\alpha\cap \alpha$. There exist $B\in U$ and $m \subset \theta$ such that for any $\alpha\in B$, $m\cap \alpha = m_\alpha\cap \alpha$. Let $C\subset \theta$ be the club consisting of ordinals $\alpha$ such that for any $\beta<\alpha$, $\sup m_\beta < \alpha$. It can be checked easily that $\{m_\alpha: \alpha\in C\cap B\}$ forms a $\Delta$-system.
\end{proof}

\begin{lemma}\label{lemma:deltasystem}
Given a sequence $\langle f_\alpha: \alpha<\theta\rangle$ such that for each $\alpha$, $f_\alpha\in \Pi_{i<\alpha} Add(\kappa_i^+, \lambda)$, there exists $B\in U$ such that for any $\gamma\in \theta$, $\{f_\alpha(\gamma): \alpha\in B-(\gamma+1)\}$ forms a $\Delta$-system.
\end{lemma}

\begin{proof}
For each $\gamma<\theta$, apply Lemma \ref{lemma:singleinstance} to find $B_\gamma\in U$ such that $\{f_\alpha(\gamma): \alpha\in B_\gamma\}$ forms a $\Delta$-system. Let $B=\Delta_{\gamma<\theta} B_\gamma$, which satisfies the conclusion as desired.
\end{proof}

Let $G\subset \po_{\vec{E}}$ be generic. In $V[G]$, let $\bar{U}$ be the filter generated by $U$. Namely, $\bar{U}=\{B\subset \theta: \exists A\in U, A\subset B\}$.

\begin{lemma}\label{lemma:completeness}
$\bar{U}$ is $\theta$-complete and normal.
\end{lemma}

\begin{proof}
Let $\langle A_\alpha\in \bar{U}: \alpha<\theta\rangle$ be a given sequence. We may without loss of generality assume each $A_\alpha\in U$ by the definition of $\bar{U}$. By Proposition \ref{proposition:almostSacks}, we know there exists $\{\mathcal{A}_\alpha\in [U]^{<\theta}: \alpha<\theta\}\in V$ such that for each $\alpha<\theta$, $A_\alpha\in \mathcal{A}_\alpha$. By $\theta$-completeness of $U$ in $V$, we can let $B_\alpha=\bigcap \mathcal{A}_\alpha\in U$ and we know that $B_\alpha\subset A_\alpha$. By the normality of $U$ in $V$, $\Delta_{\alpha<\theta} B_\alpha\in U$. It is clear that $\Delta_{\alpha<\theta} B_\alpha \subset \Delta_{\alpha<\theta} A_\alpha$, implying that $\Delta_{\alpha<\theta} A_\alpha\in \bar{U}$. Hence $\bar{U}$ is normal in $V[G]$.
\end{proof}

\begin{lemma}\label{lemma:positivesets}
Suppose we are given
\begin{enumerate}
\item $p\Vdash \dot{S}\in \bar{U}^+$ and
\item for each $\beta\in \lim\theta$ and $\nu\in A^p_\beta$, a dense open set $D_{\nu, <\beta}\subset \pi^*_{\nu}(\po_{\vec{E}\restriction \beta})$.
\end{enumerate}
Then there exists a direct extension $q\geq^* p$, $B\in U$ and $\langle p_\beta\in \po_{\vec{E}\restriction \beta}: \beta\in B\rangle$ such that for each $\beta\in B$, 
\begin{enumerate}
\item $p_\beta\geq_{\po_{\vec{E}\restriction \beta}} q\restriction \beta$,
\item $p_\beta\fr q\downharpoonright \beta \in \po_{\vec{E}}$ forces that $\beta\in \dot{S}$ and
\item for each $\nu\in A^q_\beta$, $\pi^*_\nu(p_\beta)\in D_{\nu,<\beta}$.
\end{enumerate}
\end{lemma}

\begin{proof}
For the sake of simplicity, we assume $s(p)=\emptyset$ and $\Vdash \dot{S}$ consists of limit ordinals.
We construct recursively $\langle q_\beta\in \po_{\vec{E}}: \beta<\theta\rangle$, $\langle p_\beta\in \po_{\vec{E}\restriction \beta}: \beta\in \lim \theta\rangle$ satisfying that
\begin{enumerate}
\item for any $\beta_0<\beta_1$, $q_{\beta_1}\restriction \beta_0=q_{\beta_0}\restriction \beta_0$ and $q_{\beta_1}\geq^* q_{\beta_0}$,
\item for any $\beta$, $p_\beta\geq_{\po_{\vec{E}\restriction \beta}} q_\beta\restriction \beta$ and $p_\beta \fr q_\beta \downharpoonright \beta \in \po_{\vec{E}}$,
\item if $q_\beta\not\Vdash \beta\not\in \dot{S}$, then $p_\beta\fr q_\beta\downharpoonright \beta \Vdash \beta\in \dot{S}$,
\item for any $\nu\in A^{q_\beta}_\beta$, $\pi^*_{\nu}(p_\beta)\in D_{\nu, <\beta}$.
\end{enumerate}

Applying Lemma \ref{lemma:NameReduction} and directly extending $p$ if necessarily, we may assume that for any $\nu\in A^q_\beta$, $p^{+\nu}$ forces that the Boolean value deciding $``\beta\in \dot{S}"$ is a $\pi^*_\nu(\po_{\vec{E}\restriction \beta})$-name. Suppose we are at stage $\beta$ of the construction, having already defined $\langle q_{\alpha}: \alpha<\beta\rangle$ and $\langle p_\alpha: \alpha<\beta\rangle$ satisfying the requirements above. Let us now define $q_\beta$. If $\beta$ is a successor, then $q_\beta$ is $q_{\beta-1}$. If $\beta$ is a limit, then $q^*$ is defined to be the condition such that $q^*\restriction \alpha=q_\alpha\restriction \alpha$ for all $\alpha<\beta$ and for $\gamma\geq \beta$, $q^*(\gamma)$ is the least upper bound of $\langle q_\alpha(\gamma): \alpha<\beta\rangle$. For each $\nu\in A^{q^*}_{\beta}$, if there exists a condition extending $\pi^*_\nu (q^*\restriction \beta)$ in $D_{\nu,<\beta}$ that forces in $\pi^*_\nu(\po_{\vec{E}\restriction \beta})$ that $\beta\in \dot{S}$, then we define $p_{\beta, \nu}$ to be this condition and call $\nu$ \emph{successful}. Otherwise, $p_{\beta,\nu}$ is any extension of $\pi^*_{\nu} (q^*\restriction \beta)$ in $D_{\nu,<\beta}$ and $\nu$ is \emph{unsuccessful}. Next we directly extend $q^*$ to $q_\beta$ such that 
\begin{itemize}
\item $q_\beta\restriction \beta = q^*\restriction \beta$,
\item either for all $\nu \in A^{q_{\beta}}_{\beta}$, $\nu$ is successful or for all $\nu \in A^{q_\beta}_\beta$, $\nu$ is unsuccessful,
\item there is $p_\beta\in \po_{\vec{E}}$ such that $p_\beta\fr q_\beta\downharpoonright \beta \in \po_{\vec{E}}$ and for each $\nu\in A^{q_\beta}_\beta$, $\pi^*_\nu(p_\beta)=p_{\beta,\nu}$.
\end{itemize}
If for all $\nu \in A^{q_\beta}_\beta$, $\nu$ is successful, then we say that $\beta$ is \emph{successful}. Otherwise, $\beta$ is \emph{unsuccessful}. This finishes the definition of $q_\beta$ and $p_\beta$. We need to verify the requirements are met. It is clear from the definition that we only need to check (3). If $q_\beta\not\Vdash \beta\not\in \dot{S}$, then $\beta$ is successful. Since if $\beta$ is unsuccessful, then by the definition of $q_\beta$, we have for all $\nu\in A^{q_\beta}_\beta$, $q_\beta^{+\nu}\Vdash \beta\not\in \dot{S}$. Hence $q_\beta\Vdash \beta\not\in \dot{S}$, which is impossible. As a result, $\beta$ is successful, then for any $\nu\in A^{q_\beta}_\beta$, $(p_\beta \fr q_\beta\downharpoonright \beta)^{+\nu}$ forces that $\beta\in \dot{S}$. As a result, $p_\beta \fr q_\beta \downharpoonright \beta \Vdash \beta\in \dot{S}$.

Finally, let $q=\lim_\beta q_\beta$. More precisely, $q(\beta)=q_{\beta}(\beta)$ for each $\beta<\theta$. Then it is clear that for each limit $\beta$, $q\geq q_\beta$, $q\restriction \beta = q_\beta\restriction \beta$ and hence $p_\beta \geq_{\po_{\vec{E}\restriction \beta}}q\restriction \beta$.

It is only left to check that $B=\{\beta<\theta: \beta \text{ is successful}\}\in U$. Suppose not, then $B^c\in U$. We show that $q\Vdash B^c\cap \dot{S}=\emptyset$, which contradicts with the assumption that $q\Vdash \dot{S}\in \bar{U}^+$. To see this, note that for each $\beta\in F^c\cap \lim\theta$, as $\beta$ is unsuccessful, we have $q_\beta\Vdash \beta\not\in \dot{S}$. Since $q\geq q_\beta$, $q\Vdash \beta\not\in \dot{S}$.
\end{proof}

The next proposition demonstrates the ineffability of $\theta$ relative to sequences of sets individually lying in the ground model (not necessarily the whole sequence) retains in the extension.

\begin{proposition}\label{prop:V-ineffability}
In $V[G]$, the following property holds: for any $S\in \bar{U}^+$, any $\langle X_\alpha\in V\cap P(\alpha): \alpha\in S\rangle$, there exists $X\subset \theta$ and $X\in V$ such that $\{\alpha\in S: X\cap \alpha=X_\alpha\}\in \bar{U}^+$.
\end{proposition}
\begin{proof}
Given $p\in \po_{\vec{E}}$ and a name $\langle \dot{X}_\alpha: \alpha<\theta\rangle$ as in the hypothesis, we will find an extension of $p$ forcing that ``there is some $X\subset \theta$ in $V$ such that $\{\alpha: X\cap \alpha=\dot{X}_\alpha\}$ is in $\bar{U}^+$". For simplicity, we assume that $s(p)=\emptyset$. By Lemma \ref{lemma:NameReduction} and Lemma \ref{lemma:positivesets}, we may assume that there are $q\geq^* p$, $B\in U$ and $\langle p_\beta\in \po_{\vec{E}\restriction \beta}: \beta\in B\rangle$, such that 
\begin{enumerate}
\item $p_\beta\geq_{\po_{\vec{E}\restriction \beta}}q\restriction \beta$,
\item $p_\beta \fr q\downharpoonright \beta$ forces that $\beta\in \dot{S}$, 
\item for each $\beta\in B$ and $\nu\in A^q_\beta$, $p^{+\nu}$ forces that $\dot{X}_\beta=X_\nu^\beta$ for some $X_\nu^\beta\in V\cap P(\beta)$.
\end{enumerate}

For each $\beta<\theta$, there exists $X_\beta\subset \beta$ such that the collection of $\nu\in A^q_\beta$ with $X_\beta = X_\nu^\beta$ is in $E_\beta(f^q_\beta)$,
since $E_\beta(f^q_\beta)$ is $\kappa_\beta$-complete and $2^\beta < \kappa_\beta$. Extending $q$ if necessary, we may assume $q$ already satisfies the property.

Shrink $B$ to $B'\in U$ and find $X\subset \theta$, $s\in [\theta]^{<\omega}$ and $t$ such that 
\begin{itemize}
\item for any $\beta\in B'$, $X\cap \beta = X_\beta$,
\item for any $\beta\in B'$, $s(p_\beta)=s$ and $p_\beta\restriction \max s =t$,
\item for any $\gamma<\theta$, $\{f^{p_\beta}_\gamma: \beta\in B'-(\gamma+1)\}$ forms a $\Delta$-system with root $f_\gamma$ (apply Lemma \ref{lemma:deltasystem}).
\end{itemize}

We can further extend $q$ to $q'\in \po_{\vec{E}}$ such that 
\begin{enumerate}
\item $s(q')=s$,
\item $q'\restriction \max s =t$,
\item for any $\gamma\in \theta-\max s$, $f^{q'}_\gamma = f_\gamma$.
\end{enumerate}

\begin{claim}\label{claim: furtherpositive}
No direct extension of $q'$ can force that $\{\beta\in \dot{S}: p_\beta\fr q\downharpoonright\beta\in \dot{G}\}$ is disjoint from some $E\in U$.
\end{claim}
\begin{proof}[Proof of the Claim]
Suppose for the sake of contradiction that $q^*$ is a direct extension of $q'$ forcing that there is some $E^*\in U$ disjoint from $\{\beta: p_\beta\fr q\downharpoonright\beta\in \dot{G}\}$. Apply Proposition \ref{proposition:almostSacks}, we can find $q''\geq^* q^*$ and $\mathcal{E}\in [U]^{<\theta}$ such that $q''$ forces that $E^*\in \mathcal{E}$. Since $U$ is $\theta$-complete, we know that $E=\cap \mathcal{E}\in U$. In other words, we have found a direct extension $q''$ of $q'$ such that $q''\Vdash E$ is disjoint from $\{\beta: p_\beta\fr q\downharpoonright\beta\in \dot{G}\}$.

For each $\gamma\in \theta-\max (s)$, there exists $g(\gamma)$ such that for all $\alpha>g(\gamma)$ and $\alpha\in B'$, it is the case that $f^{q''}_\gamma$ and $f^{p_{\alpha}}_\gamma$ are compatible as functions. Let $C\subset \theta$ be the set of closure points of $g$. Let $\beta\in E\cap B'\cap C$. Then $q''$ and $p_\beta\fr q \downharpoonright \beta$ are compatible since: 
\begin{itemize}
\item $q''\downharpoonright \beta\geq^* q\downharpoonright \beta$,
\item $q''\restriction \max(s)\geq^* q\restriction \max(s)=t=p_\beta\restriction \max (s)$,
\item for $\gamma\in (\max (s), \beta)$, $f^{p\beta}_\gamma$ is compatible with $f^{q''}_\gamma$ since $\beta\in C\cap B'$.
\end{itemize}
Any common extension of $q''$ and $p_\beta\fr q\downharpoonright\beta$ forces that $\beta\in \{\alpha \in \dot{S} : p_\alpha\fr q\downharpoonright\alpha \in \dot{G}\}\cap E\neq\emptyset$, which is a contradiction.
\end{proof}

By Claim \ref{claim: furtherpositive} and Corollary \ref{corollary:PrikryLemma}, we can take a direct extension $q^*$ of $q$ forcing that $\{\beta:p_\beta\fr q\downharpoonright\beta \in \dot{G}\}=\{\beta: X\cap \alpha=\dot{X}_\alpha\}\in \bar{U}^+$.\end{proof}

\begin{corollary}
$\po_{\vec{E}}$ does not add a fresh subset of $\theta$.
\end{corollary}

\begin{remark}
$\po_{\vec{E}}$ does add a fresh subset of $\theta^+$. In fact, it adds $\lambda$ many $\theta^+$-Cohen sets.
\end{remark}

\begin{corollary}\label{cor: noAmenable}
In $V[G]$, for any $S\in \bar{U}^+$ and any C-sequence $\langle C_\alpha: \alpha\in S\rangle$, there exists a club $D\subset \theta$ such that $\{\alpha\in S: D\cap \alpha\subset C_\alpha\}\in \bar{U}^+$.
\end{corollary}
\begin{proof}
This follows immediately from Corollary \ref{cor:nofastclub} (3), Lemma \ref{lemma:NameReduction} and Proposition \ref{prop:V-ineffability}.
\end{proof}

\begin{remark}
The conclusion of Corollary \ref{cor: noAmenable} implies that ``there does not exist an amenable C-sequence''. The concept of an amenable C-sequence appears in \cite{MR3914943}, strengthening the principle $\otimes_{\overrightarrow{C}}$ from \cite[p. 134]{MR1318912}.
That ``there is no amenable C-sequence on $\kappa$'' is a non-trivial compactness principle. See the discussion in \cite[Section 5]{JingOmerRadin} for more details. Corollary \ref{cor: noAmenable} can be used to argue that with suitable $\vec{E}$, in $V^{\po_{\vec{E}}}$, there is no amenable C-sequence on $\theta$ while $\diamondsuit(\theta)$ fails. Such a model was first constructed in \cite{JingOmerRadin}, using a different method based on an analysis of suitable Radin extensions.
\end{remark}

Finally, we show that in $V[G]$, even though $2^\theta = \lambda$ and $\lambda$ can be large, the degree of saturation of $\bar{U}$ remains small. Assume $V\models \mathrm{GCH}$.

\begin{proposition}\label{proposition:saturation}
$\bar{U}$ is $\theta^{++}$-saturated.
\end{proposition}
\begin{proof}
Suppose we are given $\langle \dot{S}_i\in \bar{U}^+: i<\theta^{++}\rangle$ and a condition $p\in \po_{\vec{E}}$. We need to find some $q\geq p$ and $i\neq j<\theta^{++}$ such that $q\Vdash \dot{S}_i\cap \dot{S}_j\in \bar{U}^+$. For simplicity, assume that $s(p)=\emptyset$.

For each $i<\theta^{++}$, by Lemma \ref{lemma:positivesets}, we find a direct extension $q_i\geq^* p$, $B_i\in U$ and $\langle p^i_\beta\in \po_{\vec{E}\restriction \beta}: \beta\in B_i\rangle$ such that for each $\beta\in B_i$, 
\begin{enumerate}
\item $p^i_\beta\geq_{\po_{\vec{E}\restriction \beta}} q_i\restriction \beta$,
\item $p^i_\beta\fr q_i\downharpoonright \beta$ forces that $\beta\in \dot{S}_i$.
\end{enumerate} 

By shrinking each $B_i$ if necessary, we may assume
\begin{itemize}
\item there exists $s_i\in [\theta]^{<\omega}$ and $t_i$ such that for all $\beta\in B_i$, $s(p^i_\beta)=s_i$ and $p^i_\beta\restriction \max s_i = t_i$, 
\item for any $\gamma\in \theta-\max(s_i)$, $\{f^{p_\eta^i}_\gamma: \eta\in B_i-(\gamma+1)\}$ forms a $\Delta$-system with root $f^i_\gamma$ (apply Lemma \ref{lemma:deltasystem}).
\end{itemize}
We may extend $q_i$ if necessary so that 
\begin{enumerate}
\item $s(q_i)=s_i$,
\item $q_i\restriction \max s_i = t_i$, and
\item for each $\gamma\in \theta-\max(s_i)$, $f^{q_i}_\gamma=f^i_\gamma$.
\end{enumerate}

Since $2^\theta=\theta^+$ in $V$, we can find $A\in [\theta^{++}]^{\theta^{++}}$, $s\in [\theta]^{<\omega}$ and $t$ such that 
\begin{enumerate}
\item for any $i\in A$, $s_i=s$ and $t_i = t$,
\item for any $i,j\in A$, $q_i$ is compatible with $q_j$,
\item for any $i, j\in A$ and any $\beta\in B_i\cap B_j$, $p^i_\beta$ is compatible with $p^j_\beta$ in $\po_{\vec{E}\restriction \beta}$.
\end{enumerate}
Let $i<j\in A$. The following claim finishes the proof of the lemma. Let $q^*$ be the least upper bound of $p_i$ and $p_j$.
\begin{claim}
There exists a direct extension $q$ of $q^*$ forcing that $\dot{S}_i\cap \dot{S}_j\in \bar{U}^+$.
\end{claim}

\begin{proof}[Proof of the Claim]
Suppose not, then by Proposition \ref{proposition:almostSacks}, we can find a direct extension $q'$ of $q^*$ and $\mathcal{B}\in [U]^{<\theta}$ such that $q'\Vdash \exists B'\in \mathcal{B}$, $\dot{S}_i\cap \dot{S}_j\cap B=\emptyset$. Let $B=\bigcap \mathcal{B}\in U$, then we know that $q'\Vdash \dot{S}_i\cap \dot{S}_j\cap B=\emptyset$. For each $\gamma\in \theta-\max s$, there exists some $\eta_\gamma<\theta$, such that for all $\eta>\eta_\gamma$ in $B_i\cap B_j$, $f^{q'}_\gamma$, $f^{p^i_\eta}_\gamma$ and $f^{p^j_\eta}_\gamma$ are mutually compatible. To see this, note that $\dom(f^{q'}_\gamma)\supset \dom(f^{q_i}_\gamma)=\dom(f^i_\gamma)$, hence, for sufficiently large $\eta$, $\dom(f^{p^i_\eta}_\gamma)\cap \dom(f^{q'}_\gamma) \subset \dom(f^i_\gamma)$. The same reasoning works for $j$. That $f^{p^i_\eta}_\gamma$ is compatible with $f^{p^j_\eta}_\gamma$ is guaranteed by the fact that $p^i_\eta$ is compatible with $p^j_\eta$. Let $C$ be the club consisting of the closure points of the function $\gamma\mapsto \eta_\gamma$. Take $\eta\in C\cap B\cap B_i\cap B_j$. As a result, there exists a direct extension $r$ of $q'$ extending both $p^i_\eta \fr q_i\downharpoonright \eta $ and $ p^j_\eta \fr q_j\downharpoonright\eta$. Therefore, $r\Vdash \eta\in \dot{S}_i\cap \dot{S}_j\cap B$, contradicting with the assumption on $q'$ and the choice of $B$.
\end{proof}

\end{proof}

\section{Open questions}\label{Section: Questions}

\begin{question}\label{question: approxdiamondRegTOSing}
Does $\diamondsuit_{\vec{\lambda}}(Reg^\theta)$ imply $\diamondsuit_{\vec{\lambda}}(Sing^\theta_{>\omega})$ outright (see Theorem \ref{theorem: mainGCH} for relevant  notations) or under additional assumptions such as GCH or $\theta$ is weakly compact?\\
\end{question}

\begin{question}\label{question:measurableDiamondSing}
Does $\diamondsuit(Reg^\theta)$ imply $\diamondsuit(Sing^\theta)$? Does $\theta$ being measurable imply $\diamondsuit(Sing^\theta)$?\\
\end{question}

\begin{question}\label{question: approxdiamondTOdiamondWithCardArithmetic}
Is it consistent that $\neg\diamondsuit_{\vec{\mathrm{NS}},\vec{\beth}}(\theta) + \Diamond(\theta)$ in a model where $2^{\alpha}$ is regular for stationarily many $\alpha < \theta$ (See Theorem \ref{theorem: SeparatingDiamonds})?\\
\end{question}

\begin{question}\label{question: approxdiamondEverewhereToDiamond}
Does $\forall \vec{\lambda}\ \diamondsuit_{\vec{\lambda}}(\theta)$ imply $\diamondsuit(\theta)$?\footnote{Note that $\diamondsuit_{\vec{\mathrm{NS}},\langle \alpha^{++}:\alpha<\theta\rangle}(\theta)$ fails in the model described in Theorem \ref{theorem: SeparatingDiamonds}.}\\
\end{question}

\begin{question}
Is it consistent that a strongly inaccessible $\theta$ carries a $\theta^+$-saturated ideal and $\neg\diamondsuit(\theta)$? How about if $\theta$ carries a $\theta$-saturated ideal?
\end{question}

This is related to Proposition \ref{proposition:saturation}. Ketonen \cite{MR332481} showed that if $\theta$ carries a $\gamma$-saturated ideal for some $\gamma<\theta$, then $\diamondsuit(\theta)$ holds.

\bibliographystyle{alpha}
\bibliography{bib}

\end{document}